\documentclass[11pt]{article}
\usepackage{amssymb}
\usepackage{mathrsfs}
\usepackage{bbm}
\usepackage{amsmath}
\usepackage{multirow}
\usepackage{cases} %
\usepackage{graphicx}
\usepackage[dvips]{color}
\usepackage{cite}
\usepackage{color}

\renewcommand{\Box}{\framebox{\rule{0.3em}{0.0em}}}

\newtheorem{thm}{Theorem}[section]

\newtheorem{prop}{Proposition}[section]
\newtheorem{ex}{Example}[section]
\newtheorem{rem}{Remark}[section]
\newtheorem{defi}{Definition}[section]

\newtheorem{assu}{Assumption}[section]
\newtheorem{cor}{Corollary}[section]

\newcommand{\bgeqn}{\begin{eqnarray}}
\newcommand{\edeqn}{\end{eqnarray}}

\newcommand{\bgeq}{\begin{eqnarray*}}
\newcommand{\edeq}{\end{eqnarray*}}
\newcommand{\bgc}{\begin{center}}
\newcommand{\edc}{\end{center}}

\def\Re{{\rm I\!R}}

\def\cN{{\cal N }}
\def\cC{{\cal C }}
\def\cI{{\cal I }}

\def\cB{{\cal B}}

\def\cD{{\cal D}}
\def\bz{{\bar z}}

\renewcommand{\Box}{\hfill \rule{2.3mm}{2.3mm}}

\makeatletter

\@addtoreset{equation}{section} \makeatother
\newenvironment{proof}{\noindent{\bf Proof. }}{\hfill $\Box$\medskip}

\title{Necessary Optimality Conditions for Optimal Control Problems with Equilibrium Constraints}
\author{Lei Guo\thanks{Lei Guo, Sino-US Global Logistics
Institute, Shanghai Jiao Tong University, Shanghai 200030, China. E-mail: {\tt
guolayne@sjtu.edu.cn}. This author's work was supported in part by NSFC (Grant No. 11401379).} \ and \ Jane J.
Ye\thanks{Jane J. Ye, Department of Mathematics and
Statistics, University of Victoria, Victoria, BC, V8W 2Y2 Canada. E-mail:
{\tt janeye@uvic.ca}. This author's work was supported in part by  NSERC.}}

\date{March 2015, revised December 2015, April 2016}

\begin{document}
\maketitle

\baselineskip 18pt

\vspace{4pt}{\bf Abstract.} \ This paper introduces and studies the optimal control problem with equilibrium constraints (OCPEC). The OCPEC is an optimal control problem with a mixed state and control equilibrium constraint formulated as a complementarity constraint and it can be seen as a dynamic mathematical program with equilibrium constraints. It provides a powerful modeling paradigm for many practical problems such as  bilevel optimal control problems and  dynamic principal-agent problems. In this paper, we propose  weak, Clarke, Mordukhovich and strong  stationarities for the OCPEC.  Moreover, we give some sufficient conditions to ensure that the  local minimizers of the OCPEC are Fritz John type weakly stationary, Mordukhovich stationary and strongly stationary, respectively. Unlike Pontryagain's maximum principle for the classical optimal control problem with equality and inequality constraints, a counter example shows that for general OCPECs, there may exist  two sets of multipliers for the complementarity constraints. A condition under which these two sets of multipliers coincide is given.



\vspace{4pt} \noindent\vspace{4pt}{\bf Key Words.}  \ Optimal control problem with equilibrium constraints,
necessary optimality condition, weak stationarity, Clarke stationarity, Mordukhovich stationarity, strong stationarity.

\vspace{4pt}\noindent\vspace{4pt}{\bf 2010 Mathematics Subject Classification.} \  49K15, 49K21, 90C33.

\medskip

\baselineskip 16pt
\parskip 2pt
\newpage

\section{Introduction}

We are given a time interval $[t_0,t_1]\subseteq\Re$, a multifunction  $U$ mapping $[t_0,t_1]$ to nonempty subsets of $\Re^m$, and a dynamic function
$\phi: [t_0,t_1]\times \Re^n\times \Re^m\to \Re^n.$
A control or control function $u(\cdot)$ is a measurable function on $[t_0,t_1]$ such that $u(t)\in U(t)$ for almost every $t\in [t_0,t_1]$. The state or state trajectory, corresponding to a given control $u(\cdot)$, refers to a solution $x(\cdot)$ of the following controlled differential equation:
\begin{eqnarray}
&& \dot{x}(t)=\phi(t,x(t),u(t))  \mbox{ almost everywhere  (a.e.)}  \ \ t\in [t_0,t_1],\label{state equation}\\
&& (x(t_0),x(t_1))\in E,\label{boundary condition}
\end{eqnarray}
where $E$ is a given closed subset in
$\Re^n\times \Re^n$ and $\dot{x}(t)$ is the first-order derivative of the state $x(\cdot)$ at time $t$. The differential equation \eqref{state equation} linking the state $x(\cdot)$ and the control $u(\cdot)$ is referred to as the state equation.
In optimal control, one looks for a state and control pair $(x(\cdot),u(\cdot))$ satisfying the state equation \eqref{state equation} and the boundary condition \eqref{boundary condition} so as to minimize an objective function $J(x(\cdot),u(\cdot))$. In practice,  there are generally extra constraints to be satisfied by the state and control pair.  Such  constraints are called mixed state and control constraints (mixed constraints for short).

Pang and Stewart \cite{Pang08} recently  introduced a class of controlled differential variational inequality (DVI) problem as follows:
\begin{eqnarray*}
\begin{array}{l}
\dot{x}(t)=\phi(t,x(t),u(t)) \quad{\rm a.e.} \ \ t\in [t_0,t_1],\\[2pt]
 (x(t_0),x(t_1))\in E,\\[2pt]
 u(t) \in K,\ \langle u'-u(t), \Upsilon(t,x(t),u(t))\rangle \geq 0\quad  \forall u'\in K \quad{\rm a.e.} \ \ t\in [t_0,t_1],\\[2pt]
\end{array}
\end{eqnarray*}
where $\Upsilon:[t_0,t_1]\times \Re^n\times \Re^m \to \Re^m$ is a vector-valued function and $K$ is a closed convex subset in $\Re^m$. The DVI provides a powerful modeling paradigm for many practical problems such as differential Nash equilibrium games (\cite{Pang08,Chen14}),  multi-rigid-body dynamics with frictional contacts (\cite{Stewart}), and hybrid engineering systems (\cite{Heemels}). In the case where $K=\Re^m_+$, the DVI becomes the controlled differential complementarity problem (DCP)
\begin{eqnarray}\label{dcp1}
&& \dot{x}(t)=\phi(t,x(t),u(t))  \quad{\rm a.e.} \ \ t\in [t_0,t_1],\nonumber\\
&&  (x(t_0),x(t_1))\in E,\\
&& 0\leq u(t) \perp \Upsilon(t,x(t), u(t)) \geq 0  \quad{\rm a.e.} \ \ t\in [t_0,t_1],\nonumber
\end{eqnarray}
where  $a\perp b$ means that  vector $a$ is perpendicular to vector $b$.
In the case where $K$ can be expressed as a set of solutions satisfying some inequality constraints such as $$K=\{u\in \Re^m: g(u) \leq 0\}$$ where $g(\cdot)$ is a convex vector-valued function,  when $g(\cdot)$ is affine or Slater's condition holds, the DVI can also be represented as the following DCP:
\begin{eqnarray}\label{dcp2}
\begin{array}{ll}
\dot{x}(t)=\phi(t,x(t),u(t)) &{\rm a.e.} \ \ t\in [t_0,t_1],\\[2pt]
 (x(t_0),x(t_1))\in E,&\\[2pt]
-\Upsilon(t,x(t),u(t))+\nabla g(u(t)) \lambda(t)=0 &{\rm a.e.} \ \ t\in [t_0,t_1],\\[2pt]
0\leq \lambda(t) \perp -g( u(t)) \geq 0 &{\rm a.e.} \ \ t\in [t_0,t_1] ,
\end{array}
\end{eqnarray}
where $\nabla g$ denotes the transposed Jacobian of $g$ and  $\lambda(t)$ is a Lagrange multiplier corresponding to the inequality constraint $g(u(t))\leq0$.

Motivated by the studies for the DVI, we consider a class of controlled differential complementarity system where in addition to the state equation (\ref{state equation}) and the boundary condition (\ref{boundary condition}), the state and control pair $(x(\cdot),u(\cdot))$ satisfies some mixed equality and inequality constraints, as well as a mixed equilibrium system
formulated as a complementarity system:
\begin{eqnarray}\label{cpeqn}
0\leq G(t,x(t), u(t))\perp H(t,x(t), u(t))\geq0 & {\rm a.e.} \ \ t\in [t_0,t_1],
\end{eqnarray}
where $G,H:[t_0,t_1]\times \Re^{n}\times \Re^m\to \Re^l$. We say that an index $i$ is degenerate if
$
G_i(t,x(t), u(t))=H_i(t,x(t), u(t))=0.
$
It is obvious that such a system includes DCPs (\ref{dcp1}) and (\ref{dcp2}) as  special cases. Correspondingly, it is natural to determine what is the ``best" control (or the ``best''  state and control pair) satisfying such a system to achieve some given objective. A simple example is to find the best control from such a system so that   the final state $x(t_1)$ will reach some prescribed target from a given initial state $x(t_0)$. In this paper, we introduce a class of optimal control problems with equilibrium constraints (OCPEC) in which one looks for a state and control pair $(x(\cdot),u(\cdot))$ from such a system so as to minimize an objective function $J(x(\cdot),u(\cdot))$. Mathematically, the OCPEC considered in this paper is of the form
\begin{eqnarray}
{\rm (OCPEC)}~~~~\min &&  J(x(\cdot),u(\cdot)) \nonumber \\
{\rm s.t.} && \dot{x}(t)=\phi(t,x(t), u(t)) \,\quad{\rm a.e.} \  \,t\in [t_0,t_1],\nonumber\\
&& g(t,x(t), u(t))\leq0,\ h(t,x(t), u(t))=0\quad{\rm a.e.} \  \,t\in [t_0,t_1],\label{cons}\\
&& 0\leq G(t,x(t), u(t))\perp H(t,x(t), u(t))\geq0\quad{\rm a.e.} \  \,t\in [t_0,t_1],\nonumber\\
&& u(t) \in U(t) \quad{\rm a.e.} \  \,t\in [t_0,t_1],\label{control}\\
&& (x(t_0), x(t_1))\in E,\nonumber
\end{eqnarray}
where
$g:[t_0,t_1]\times \Re^{n}\times \Re^m\to \Re^{l_1}$ and $h:[t_0,t_1]\times \Re^{n}\times \Re^m\to \Re^{l_2}$.

The OCPEC can be considered as a dynamic version of the mathematical program with equilibrium constraints (MPEC) that has been an active area of research in recent years (see, e.g., the monographs \cite{Luo-book,Outrata-book}). The OCPEC provides a powerful modeling paradigm for many practical problems such as the dynamic optimization of chemical processes with changes in the number of equilibrium phases  \cite{Raghunathan04}. A large part of source problems of the OCPEC comes from  bilevel optimal control problems (see, e.g., \cite{Benita-Mehlitz,Hatz,Hatz12,Ye95,Ye97}),  Stackelberg differential games (see, e.g., \cite{Xie97,He07}), and dynamic principal-agent problems (see, e.g., \cite{continuousPA08,DPA10}) when there exist inequality constraints in the lower-level problem. For those problems, if the lower-level problem, which is a parametric optimal control problem, is replaced by Pontryagain's maximum principle (see \cite{Pontray,Vinter}) which is the well-known first-order necessary optimality condition for optimal control problems, then an OCPEC results; see, e.g., \cite[Section 6.1]{Hatz}.

It is desirable to know whether there exists an optimal control before solving the OCPEC. The Filippov's existence theorem  for Mayer's problem that is due to Filippov \cite{Filippov} (see also \cite[Theorem 9.2.i]{Cesari}) requires the convexity of the velocity set
$\phi(t,x,{\cal U}(t,x))$
where $${\cal U}(t,x):=\{u\in U(t): g(t,x,u)\leq 0, h(t,x,u)=0, 0\leq G(t,x, u)\perp H(t,x,u)\geq 0\}. $$
The velocity set is in general nonconvex due to the existence of the complementarity constraints. Thus, the classical existence theorem may not be applicable and one may need to look for new ways to prove the existence of optimal controls for the OCPEC or use the existence theorem in a relaxed control setting (\cite{Young,Warga}). We leave these challenging questions for future research.

In this paper, we assume that an optimal control exists for the OCPEC and focus on deriving its necessary optimality conditions. To the best of our knowledge, there is no such result in the literature so far. Although deriving  necessary optimality conditions for optimal control problems with mixed constraints is a highly challenging problem, some progresses have been made; see, e.g., \cite{Hestenes,Hestenes65,dub,pinho5,pinho6,C05,C-P,Li-Ye}. Unfortunately, none of these results are applicable to the OCPEC and its reformulations.
The constraint (\ref{cpeqn}) is obviously equivalent to that for almost every $t\in[t_0,t_1]$,
\begin{eqnarray}\label{cpcp}
G(t,x(t),u(t))\geq0,\ H(t,x(t),u(t))\geq0,\ G(t,x(t), u(t))^\top H(t,x(t), u(t))\leq0,
\end{eqnarray}
where $^\top$ denotes the transpose, which is clearly a system of inequalities.
However, all the inequalities in (\ref{cpcp}) never hold strictly at the same time. This means that the Mangasarian-Fromovitz constraint qualification (MFCQ) is violated at any point satisfying (\ref{cpcp}). The classical necessary optimality conditions for optimal control problems with mixed equality and inequality constraints generally require the linear independence constraint qualification  (LICQ) (see, e.g., \cite{Hestenes65}) or the Mangasarian-Fromovitz condition (MFC) (see, e.g., \cite{C-P}) over some neighborhood of the local minimizer. But both LICQ and MFC are  stronger than MFCQ. Thus, the classical necessary optimality conditions for optimal control problems with equality and inequality constraints cannot be applied to the OCPEC with the complementarity constraint (\ref{cpeqn})   reformulated as inequality constraints (\ref{cpcp}). In the MPEC literature, by using the so-called ``piecewise programming'' approach (see, e.g., \cite{Luo-book,Ye99}), the feasible region of an MPEC is locally reformulated as a union of finitely many pieces where each piece is expressed as a system of equality and inequality constraints, and then it can be shown that the strong (S-) stationarity holds under the so-called MPEC LICQ. It is obvious that such an approach fails for the dynamic complementarity system (\ref{cpeqn}). A well-known technique to derive a necessary optimality condition for an MPEC called the Clarke (C-) stationarity is to use the equivalent nonsmooth reformulation $\min\{G, H\}=0$ ($``\min"$ denotes the componentwise minimum) to replace the complementarity system $0\leq G\perp H\geq 0$ (see, e.g., \cite{Scholtes-stationarity,Jane-jmaa2005}).  This technique, however, is also not applicable to the OCPEC since such an approach leads to an optimal control problem with a  {\it nonsmooth} mixed equality constraint for which there does not exist any applicable necessary optimality conditions in the control literature.
Another equivalent reformulation of the complementarity constraint is
$\big(G, H\big)\in {\cal C}^l$ where
\begin{equation}\label{defi C}
\cC^l:=\{(a,b)\in \Re^l\times \Re^l: 0\leq a\ \bot\ b\geq0\}
\end{equation}
is called the complementarity cone.
It is known that this reformulation is useful to obtain  a necessary optimality condition in the form of Mordukhovich (M-) stationarity in the MPEC literature; see, e.g., \cite{Jane-jmaa2005}.
Using this reformulation, the OCPEC can be equivalently reformulated as
\begin{eqnarray}\label{general model}
(P_s)~~~~~~ \min &&J(x(\cdot),u(\cdot)) \nonumber\\
{\rm s.t.} &&  \dot{x}(t)=\phi(t,x(t), u(t)) \,\quad{\rm a.e.} \ \ t\in [t_0,t_1],\\
&&  (x(t),u(t))\in S(t)\,\quad{\rm a.e.} \ \ t\in [t_0,t_1], \nonumber \\
&& (x(t_0), x(t_1))\in E, \nonumber
\end{eqnarray}
with
\begin{eqnarray}\label{cpc1}
S(t):=\left\{(x,u)\in\Re^n\times U(t): \begin{array}{l}g(t,x,u)\leq0, h(t,x,u)=0\\[4pt]\big(G(t,x,u), H(t,x,u)\big)\in {\cal C}^l\end{array}\right\}.
\end{eqnarray}
An optimal control problem in the form of $(P_s)$ with an abstract mixed constraint $S(t)$ was recently studied by  Clarke and De Pinho \cite{C-P}. In this paper, we first derive a slightly sharper necessary optimality condition for $(P_s)$ than \cite[Theorem 2.1]{C-P} and then apply it to the problem with $S(t)$ defined as in (\ref{cpc1}). We hope that we would get the M-stationarity as in the MPEC literature. Unfortunately, for the OCPEC, no sign information on the multipliers associated with the degenerate indices can be derived and, consequently, we can only obtain a weak stationarity. In order to get more sign information on the multipliers associated with the degenerate indices, we further utilize the Weierstrass condition to obtain the second set of multipliers. A counter example shows that in general these two sets of multipliers may be different in measure. However, under the  MPEC LICQ, since the multipliers corresponding to the weak stationarity are unique, these two sets of multipliers coincide almost everywhere and then we can obtain the S-stationarity with one set of multipliers.

The rest of this paper is organized as follows. In Section 2, we give some preliminaries and preliminary results. In Section 3,  we develop the necessary optimality conditions for the OCPEC. Section 4 illustrates our derived results with a simple example.

\section{Preliminary  and preliminary results}\label{section2}

Throughout this paper, $\|\cdot\|$ denotes the Euclidean norm and ${\cal B}_\delta(x):=\{y: \|y-x\|<\delta\}$ the open ball centered at $x$ with positive radius $\delta$. The boundary, closure, convex hull, and  closed convex hull of a subset $\Omega\subseteq\Re^n$ are denoted by ${\rm bd}\,\Omega$, ${\rm cl}\, \Omega$,  ${\rm co}\,\Omega$, and ${\rm clco}\,\Omega$, respectively. Moreover, ${\rm dist}_\Omega (x)$ denotes the Euclidean distance from
$x$ to $\Omega$. For any $a,b\in\Re^n$, $a_+:=\max\{a,0\}$ denotes the nonnegative part of vector $a$ and  $\langle a, b\rangle$  the inner product of vector $a$ and  vector $b$. Given a mapping $\psi:\Re^n \to \Re^m$ and a point $x\in \Re^n$,
$\nabla \psi(x)$ stands for the transposed Jacobian of $\psi(\cdot)$ at $x$ and $I_\psi(x):=\{i: \psi_i(x)=0\}$ the active index set of $\psi(\cdot)$ at $x$. The Minkowski sum of a
singleton $\{a\}$ and an arbitrary set $A$ is denoted by $a$ + $A$.

\subsection{Background in variational analysis}
In this subsection, we review some basic concepts and results in variational
analysis that will be used later on; see, e.g., \cite{c,m1,RW} for more details.
Given a subset $\Omega\subseteq\Re^n$ and $x\in {\rm cl\,}\Omega$, the proximal normal cone to $\Omega$ at $x$ is defined as
\[
\cN^P_{\Omega}(x):=\{v\in \Re^n: \exists \ \sigma\geq 0 \mbox{ s.t. }   \langle v, y-x\rangle  \leq \sigma\|y-x\|^2\quad \forall y\in\Omega\} ,
\]
the limiting normal cone to $\Omega$ at $x$ is defined as
\[
\cN^L_{\Omega}(x):=\{v\in \Re^n: \exists (x^k, v^k) \to (x, v)\ {\rm with}\ v^k\in \cN^P_{\Omega}(x^k)\quad \forall k\},
\]
and the Clarke normal cone to $\Omega$ at $x$ is defined as
$
\cN^C_{\Omega}(x):={\rm clco}\,\cN^L_{\Omega}(x),
$
which also holds true even if the space is not finite dimensional but a more general Asplund space \cite{m1}. We can easily obtain the following inclusions:
\[
\cN^P_{\Omega}(x)\subseteq \cN^L_{\Omega}(x)\subseteq\cN^C_{\Omega}(x)\quad \forall x\in {\rm cl\,}\Omega.
\]
For a multifunction $\Xi:\Re^n\rightrightarrows
\Re^m$, its graph and domain are defined, respectively, as
\[
{\rm gph\, \Xi}:=\{(x,u): u\in \Xi(x)\}\quad {\rm and}\quad {\rm dom}\,\Xi:=\{x: \Xi(x)\neq \emptyset\}.
\]
Both the limiting normal cone mapping $\cN_\Omega^L(\cdot)$ and Clarke normal cone mapping $\cN_\Omega^C(\cdot)$ are closed in the sense that their graphs are closed.

The following expression for the limiting normal cone of the complementarity cone $\cC^l$ is well-known (see, e.g., \cite[Proposition 3.7]{jane00}) and will be used in  Section 3.

\begin{prop} \label{normalcone}
For any $(a,b)\in\cC^l$ where $\cC^l$ is defined in \eqref{defi C},
\begin{equation*}
\cN^L_{\cC^l}(a,b)=\left\{(\alpha,\beta)\in \Re^l\times \Re^l:\begin{array}{l}\alpha_i=0\quad {\rm if}\ a_i>0,\quad \beta_i=0\quad {\rm if}\ b_i>0\\[2pt] \alpha_i<0,\beta_i<0 {\ \rm or \ } \alpha_i\beta_i=0 \quad {\rm if}\ a_i=b_i=0\end{array}\right\}.
\end{equation*}
\end{prop}

Given a lower semicontinuous function $\varphi:\Re^n\to \Re\cup\{+\infty\}$ and a point $x$ with $\varphi(x)$ finite, the limiting subdifferential of $\varphi$ at $x$ is defined as
\[
\partial^L \varphi(x):=\left\{v\in \Re^n: \exists (x^k,  v^k ) \to (x, v)\ {\rm with}\ \lim_{y\to x^k}\frac{f(y)-f(x^k)-\langle v^k, y-x^k\rangle}{\|y-x^k\|}\geq0\ \forall k\right\}.
\]
If $\varphi(\cdot)$ is Lipschitz continuous near $x$, then the Clarke subdifferential of $\varphi(\cdot)$ at $x$ can be defined as
$
\partial^C\varphi(x):={\rm clco}\,\partial^L \varphi(x),
$
which also holds true even if the space is not finite dimensional but a more general Asplund space \cite{m1}.
Both the limiting subdifferential mapping $\partial^L \varphi(\cdot)$ and Clarke subdifferential mapping $\partial^C \varphi(\cdot)$  are closed in the sense that their graphs are closed.

Given a point $(x,u)\in {\rm cl\,gph }\Xi$ for a multifunction $\Xi:\Re^n\rightrightarrows
\Re^m$, the coderivative $D^*\Xi(x,u):\Re^m\rightrightarrows\Re^n$ of $\Xi(\cdot)$ at $(x,u)$ is defined as
\[
D^*\Xi(x,u)(y):=\{v\in \Re^n: (v,-y)\in \cN^L_{\rm gph \Xi}(x,u)\}.
\]
The symbol $D^*\Xi (x)$ is used when $\Xi(\cdot)$ is single-valued at $x$ and $u=\Xi(x)$. Moreover, if $\Xi(\cdot)$ is single-valued and Lipschitz continuous near $x$, then, by \cite[Theorem 1.90]{m1},
\[
D^*\Xi (x)(y) = \partial^L \langle y, \Xi(x)\rangle\quad \forall y\in \Re^m.
\]


\subsection{Local error bound condition and constraint qualifications}

In this subsection, we consider the following constrained system:
\begin{equation}
\Omega:=\{z\in {\cal D}:   g(z) \leq 0,\ h(z) =0,\ (G(z),H(z))\in \cC^l\},
\label{constraineds}
\end{equation}
where ${\cal D}$ is a closed subset in $\Re^d$, and $g:\Re^d\rightarrow \Re^{l_1}, h:\Re^d\rightarrow \Re^{l_2}, G,H: \Re^d\rightarrow \Re^l$ are all strictly differentiable.  We say that the local error bound condition holds (for the constrained system representing the set $\Omega$ as in \eqref{constraineds}) at $\bar{z}\in \Omega$  if there exist $\tau>0$ and $\delta>0$ such that
\begin{eqnarray*}
{\rm dist}_{\Omega}(z) \leq \tau \big( \|g(z)_+\|+\|h(z)\|+{\rm dist}_{\cC^l}(G(z),H(z))\big)\quad \forall  z\in \cB_\delta(\bar{z})\cap {\cal D}.
\end{eqnarray*}
It is well-known that the local error bound condition at $\bar z$ is equivalent to the calmness of the perturbed constrained system
\begin{equation}\label{pconstraineds}
\Omega(y^g,y^h,y^G,y^H):=\{z\in \cD: g(z)+y^g \leq 0,\ h(z)+y^h =0,\ 0\leq G(z)+y^G\perp H(z)+y^H \geq 0\}
\end{equation}
at $(0,0,0,0, \bar{z})$ (see, e.g., \cite{henrion-outrata}).
The local error bound condition is very weak  and there exist many sufficient conditions for it to hold; see, e.g., \cite{wu-ye02,Wu-Yemp, Wu-Ye03,henrion-outrata,guoyezhang13,Jane-Jin}. The following constraint qualifications are such sufficient conditions.

\begin{defi}[MPEC constraint qualifications]\label{defi}\rm
Let $\bar{z}\in\Omega$ where $\Omega$ is defined in \eqref{constraineds}. When $\cD=\Re^d$, we say that the MPEC LICQ holds at $\bz$ if the family of gradients
\[
\{\nabla g_i(\bar{z}):i\in{I}_g(\bar{z})\}\cup \{\nabla h_i(\bar{z}):i=1,\ldots,l_2\}\cup\{\nabla G_i(\bz):i\in I_G(\bz)\}\cup \{\nabla H_i(\bz): i\in I_H(\bz)\}
\]
is linearly independent.

We say that the MPEC linear condition holds if all the functions $g(\cdot),h(\cdot),G(\cdot),H(\cdot)$ are affine and $\cD$ is a union of finitely many polyhedral sets.

We say that the MPEC quasi-normality holds at $\bar{z}$ if there is no nonzero vector $(\lambda,\upsilon,\mu,\nu)$ such that
\begin{itemize}
\item  $0\in\nabla g(\bar{z})\lambda+\nabla h(\bar{z})\upsilon-\nabla G(\bz)\mu-\nabla H(\bz)\nu+\cN^L_\cD(\bz)$,
\item  $\lambda\geq0$, $\lambda_i=0\ \forall i\notin I_g(\bz)$, $\mu_i=0\ \forall i\notin I_G(\bz),\ \nu_i=0\ \forall i\notin I_H(\bz),\ \mu_i>0, \nu_i>0 \mbox{\ \rm or\ } \mu_i\nu_i=0\ \forall i\in I_G(\bz)\cap I_H(\bz)$,
\item
there exists a sequence $\{z^k\}\subseteq \cD$ converging to $\bz$ such that for
each $k$,
\begin{eqnarray*}
&& \lambda_i>0 \ \Longrightarrow \ g_i(z^k)>0,\quad \upsilon_i\neq0 \ \Longrightarrow \ \upsilon_ih_i(z^k)>0,\\
&&\mu_i\neq0  \ \Longrightarrow \ \mu_iG_i(z^k)<0, \quad \nu_i\neq0  \ \Longrightarrow \  \nu_iH_i(z^k)<0.
\end{eqnarray*}
\end{itemize}
\end{defi}

It should be noted that the MPEC quasi-normality is a weak condition which holds automatically when the MPEC linear condition holds with $\cD=\Re^d$ and is also implied by the MPEC LICQ.

\begin{prop}\label{error bound}
The local error bound condition holds at $\bz\in \Omega$ if the MPEC linear condition or the MPEC quasi-normality holds at $\bz$.
\end{prop}
\begin{proof}
If the MPEC linear condition holds, then it is easy to see that the perturbed constrained system $\Omega(y^g,y^h,y^G,y^H)$ defined in \eqref{pconstraineds} is a polyhedral multifunction and hence the local error bound condition holds \cite{robinson81}. Moreover, by \cite[Theorem 5.2]{guoyezhang13}, the local error bound condition follows from the MPEC quasi-normality immediately.
\end{proof}

\subsection{Optimal control problem with an abstract set constraint}

In this subsection, we consider the optimal control problem $(P_s)$ where
$$J(x(\cdot),u(\cdot)):=\int_{t_0}^{t_1} F(t,x(t),u(t))dt +f(x(t_0),x(t_1)).$$
Here $F:[t_0,t_1]\times \Re^n\times \Re^m\to \Re$ and $f:\Re^n\times \Re^n\to \Re$.
The basic hypotheses on the problem data, in force throughout this subsection, are the following: $F(\cdot), \phi(\cdot)$ are ${\cal L}\times {\cal B}$ measurable, $S(\cdot)$ is ${\cal L}$ measurable, and $f(\cdot)$ is locally Lipschitz continuous, where ${\cal L}\times {\cal B}$ denotes the $\sigma$-algebra of subsets of appropriate spaces generated by product sets $M\times N$ where $M$ is a Lebesgue $({\cal L})$ measurable subset in $\Re$ and $N$ is a Borel $({\cal B})$ measurable subset in $\Re^n\times \Re^m$.


We refer to any absolutely continuous function as  an arc.
An admissible pair for $(P_s)$  is a pair of functions $(x(\cdot), u(\cdot))$ on $[t_0,t_1]$ for which $u(\cdot)$ is a control and $x(\cdot)$ is an arc that satisfies all the constraints in $(P_s)$.
Given a measurable radius function $R:[t_0,t_1]\to (0,+\infty]$, as in  \cite{C-P}, we say that an admissible pair $(x^*(\cdot),u^*(\cdot))$ is a local minimizer of radius $R(\cdot)$ for $(P_s)$ if there exists $\epsilon >0$ such that for every pair $(x(\cdot),u(\cdot))$ admissible for $(P_s)$ which also satisfies
$$
\|x(t)-x^*(t)\| \leq \epsilon,\quad \|u(t)-u^*(t)\|\leq R(t)\ \ {\rm  a.e.} \ t\in[t_0,t_1],\quad  \int_{t_0}^{t_1}\|\dot{x}(t)-\dot{x}^*(t)\|\,dt\leq \epsilon,
$$
we have $J(x^*(\cdot),u^*(\cdot))\leq J(x(\cdot),u(\cdot))$.  Note that the so-called $W^{1,1}$ local minimizer in the control literature is actually the case where the radius function $R(\cdot)$ is identically $+\infty$.

Let $(x^*(\cdot),u^*(\cdot))$ be a  local minimizer of radius $R(\cdot)$ for $(P_s)$.
For every given $t\in[t_0,t_1]$, a radius function $R(\cdot)$, and a positive constant $\epsilon$,  we define a neighborhood of the point $(x^*(t),u^*(t))$ as follows:
\begin{equation}
S_*^{\epsilon,R}(t):=\big\{(x,u)\in S(t): \|x-x^*(t)\|\leq\epsilon, \|u-u^*(t)\|\leq R(t)\big\}.\label{neighbor}
\end{equation}

%

Other than the basic hypotheses on the problem data, we also assume that the following assumptions hold for  $(P_s)$.

\begin{assu}\label{ass1}
{\rm (i)} There exist measurable functions $k_x^\phi(\cdot),k_x^F(\cdot),k_u^\phi(\cdot), k_u^F(\cdot)$ such that for almost every $t\in[t_0,t_1]$ and for every $(x^1,u^1),(x^2,u^2)\in S_*^{\epsilon,R}(t)$, we have
\begin{eqnarray*}
\|\phi(t,x^1,u^1)-\phi(t,x^2,u^2)\|&\leq& k_x^\phi(t)\|x^1-x^2\|+ k_u^\phi(t)\|u^1-u^2\|,\\
|F(t,x^1,u^1)-F(t,x^2,u^2)|&\leq& k_x^F(t)\|x^1-x^2\|+ k_u^F(t)\|u^1-u^2\|.
\end{eqnarray*}

{\rm (ii)} There exists a positive measurable function $k_S(\cdot)$ such that for almost every $t\in [t_0,t_1]$, the bounded slope condition holds:
\begin{eqnarray}\label{bsc}
(x,u) \in S_*^{\epsilon,R}(t),\ (\alpha,\beta)\in {\cal N}^P_{S(t)}(x,u)\Longrightarrow \|\alpha\|\leq k_S(t)\|\beta\|.
\end{eqnarray}

{\rm (iii)} The functions $k_x^\phi(\cdot),k_x^F(\cdot),k_S(\cdot)[k_u^\phi(\cdot)+k_u^F(\cdot)]$ are integrable and there exists a positive number $\eta$ such that $R(t)\geq \eta k_S(t)\ a.e.\ t\in[t_0,t_1]$.
\end{assu}

Assumption \ref{ass1}(i) can be seen as a local Lipschitz condition in variable $(x,u)$ with measurable Lipschitz constants. This condition is automatically satisfied with time independent Lipschitz constants when $u^*(\cdot)$ is bounded over $[t_0,t_1]$, the radius function $R(\cdot)$ is a positive constant function, and the functions $F(\cdot),\phi(\cdot)$ are locally Lipschitz continuous in variable $(t,x,u)$. Assumption \ref{ass1}(ii) is a key condition proposed in \cite{C-P} to derive the necessary optimality conditions. We will investigate some sufficient conditions for such an assumption to hold in our problem setting in Section 3.

%

For a general optimal differential inclusion problem
\begin{eqnarray*}
\min && f(x(t_0),x(t_1))\\
{\rm s.t.} && \dot{x}(t)\in F_t(x(t))\quad\, {\rm a.e.}\ \ [t_0,t_1],\\
           && (x(a),x(b))\in E,
\end{eqnarray*}
where $F_t:\Re^n\rightrightarrows \Re^n$ is a multifunction, Clarke \cite{C05} has derived a new state of the art necessary optimality conditions in the optimal control literature. These conditions are stratified in that both the hypotheses and the conclusions are formulated relative to a given radius function. However, it should be noted that for a point $v$ lying on the boundary of $F_t(x^*(t))\cap {\rm cl}\,\cB_{R(t)}(u^*(t))$, one may not find a sequence $\{v^k\}$ in $F_t(x^*(t))\cap \cB_{R(t)}(u^*(t))$ such that $v^k\to v$ if $F_t(x^*(t))$ is disconnected. Thus, the derived Weierstrass condition in \cite[Theorems 2.3.3 and 3.1.1, and Corollary 3.5.3]{C05} should hold only relative to the open ball  $\cB_{R(t)}(u^*(t))$ instead of the closed ball ${\rm cl}\,\cB_{R(t)}(u^*(t))$. In a recent paper \cite{b}, Bettiol et al. also proved the stratified necessary optimality conditions for an optimal differential inclusion problem involving additional pathwise state constraints in \cite[Theorem 2.1]{b} and pointed out that the Weierstrass condition may not hold with full radius by a counter example \cite[Example 2]{b}. Recently, Clarke and De Pinho \cite[Theorem 2.1]{C-P} derived the stratified necessary optimality conditions for $(P_s)$ by recasting the problem as an equivalent optimal different  inclusion problem and applying the corresponding necessary optimality conditions from  \cite[Corollary 3.5.3]{C05}.  In the following, using the same proof technique as in \cite[Theorem 2.1]{C-P}, we give a stratified necessary optimality condition for $(P_s)$ which will be used in obtaining our main results. Our results differ from \cite[Theorem 2.1]{C-P} in two aspects.  Firstly, our Euler inclusion (c) is slightly sharper  than that in \cite[Theorem 2.1]{C-P}. Secondly, the Weierstrass condition (d) holds only on the open ball  $\cB_{R(t)}(u^*(t))$ instead of the closed ball ${\rm cl}\,\cB_{R(t)}(u^*(t))$.

\begin{thm}\label{lema}
Let $(x^*(\cdot),u^*(\cdot))$ be a   local minimizer of radius $R(\cdot)$ for $(P_s)$ and Assumption \ref{ass1} hold.
Then there exist a number $\lambda_0\in\{0,1\}$ and an arc $p(\cdot)$ such that
\begin{itemize}
\item[{\rm (a)}] the nontriviality condition holds: $(\lambda_0,p(t))\neq0\quad \forall t\in [t_0,t_1]$;
\item[{\rm (b)}] the transversality condition  holds:
$$(p(t_0),-p(t_1))\in \lambda_0\partial^L f(x^*(t_0),x^*(t_1))+{\cal N}_E^L(x^*(t_0),x^*(t_1));$$
\item[{\rm (c)}] the Euler inclusion holds: For almost every $t\in[t_0,t_1]$,
\begin{eqnarray}\label{euler adjoint}
&&(\dot{p}(t),0) \in {\rm co}\big\{(w,0): (w,0)\in \partial^L\big\{\langle -p(t),\phi(t,\cdot,\cdot)\rangle+\lambda_0 F(t,\cdot,\cdot)\big\}(x^*(t),u^*(t))\nonumber\\
&&\qquad\qquad\quad +{\cal N}^L_{S(t)}(x^*(t),u^*(t))\big\};
\end{eqnarray}

\item[{\rm (d)}] the Weierstrass condition holds: For almost every $t\in[t_0,t_1]$,
\begin{equation*}
\begin{array}{l}
(x^{*}(t),u)\in S(t), \ \|u-u^*(t)\|< R(t)\quad \Longrightarrow\\[3pt]
\langle p(t),\phi(t,x^*(t),u) \rangle -\lambda_0F(t,x^*(t),u) \leq \langle p(t), \phi (t,x^*(t),u^*(t)) \rangle -\lambda_0 F(t,x^*(t),u^*(t)).
\end{array}
\end{equation*}
\end{itemize}
\end{thm}
\begin{proof}
First we consider the case where $F(\cdot)\equiv 0$.  Similarly as in the proof of \cite[Theorem 2.1]{C-P},  for any $M>1$, by applying \cite[Corollary 3.5.3]{C05} with the Weirestrass condition on a open ball $\cB_{R(t)}(u^*(t))$ or \cite[Theorem 2.1]{b}, we can obtain a number $\lambda_{0,M}\in \{0,1\}$ and an arc $p_M(\cdot)$ such that the nontriviality condition holds:
$$ \lambda_{0,M}+\|p_M(\cdot)\|_\infty =1;$$
the transversality condition holds:
 $$(p_M(t_0),-p_M(t_1))\in \lambda_{0,M}\partial^L f(x^*(t_0),x^*(t_1))+{\cal N}_E^L(x^*(t_0),x^*(t_1));$$
the Euler inclusion holds: For almost every $t\in[t_0,t_1]$,
$$(\dot{p}_M(t),0) \in {\rm co}\big\{(w,0): (w,p_M(t),0)\in {\cal N}^L_{G(t)}(x^*(t),\phi(t,x^*(t),u^*(t)),0)\big\},$$
where
\[
G(t):=\big\{(x,\phi(t,x,u),c(t)(u-u^*(t))):(x,u)\in S(t)\big\}
\]
with $c(t):=M(k_x^\phi(t)+k_S(t)k_u^\phi(t))/k_S(t)$; and the Weierstrass condition holds with radius $R(\cdot) M/(M+1)$: For almost every $t\in[t_0,t_1]$,
\begin{eqnarray}
&&(x^*(t), u) \in S(t), \|u-u^*(t)\|< R(t)M/(M+1)\nonumber\\
&&\qquad \qquad\qquad \qquad\Longrightarrow \langle p_M(t), \phi(t, x^*(t), u)\rangle \leq \langle p_M(t), \phi(t, x^*(t), u^*(t))\rangle.\label{Weiers}
\end{eqnarray}
As shown in \cite[Theorem 2.1]{C-P}, we can extract a convergent subsequence of the sequence $\{(\lambda_{0,M}, p_M(\cdot))\}_M$ with limit $(\lambda_0, p(\cdot))$  as $M\rightarrow \infty$.  Taking limits as $M\rightarrow \infty$ in the above nontriviality condition, transversality condition, Weierstrass condition, and Euler inclusion, we can obtain the results (a), (b), and (d) of this theorem for the case where $F(\cdot)\equiv 0$ and
$$
(\dot{p}(t),0) \in {\rm co}\big\{(w,0): (w,p(t),0)\in {\cal N}^L_{G(t)}(x^*(t),\phi(t,x^*(t),u^*(t)),0)\big\}.
$$
The Euler inclusion (c) of this theorem for the case where $F(\cdot)\equiv 0$ can be obtained by estimating the limiting normal cone of the above formula as in the last paragraph of Page 4521 in \cite{C-P}.

The general case in which a nonzero $F$ is present is reducible to the already treated one by augmentation as explained at the end of the proof of \cite[Theorem 2.1]{C-P}.
\end{proof}

Note that in the proof of Theorem \ref{lema}, we are unable to prove that
the Weierstrass condition holds with full radius $R(\cdot)$ as claimed in the proof of \cite[Theorem 2.1]{C-P}.  The reason is that for  a given $u$ lying on the boundary of the set
$$\Omega:=\{ u:  (x^{*}(t),u)\in S(t), \ \|u-u^*(t)\|\leq  R(t)\},$$ to show that
\begin{equation}
\langle p(t),\phi(t,x^*(t),u) \rangle  \leq \langle p(t), \phi (t,x^*(t),u^*(t)) \rangle \label{Weiers1} \end{equation} in the case where $F(\cdot)\equiv 0$,  we would need to find $u_M \in  \{ u:  (x^{*}(t),u)\in S(t), \ \|u-u^*(t)\|<R(t)M/(M+1)\}$ such that $u_M\to u$ as $M\to\infty$ and take limits in  (\ref{Weiers}) to derive the desired inequality (\ref{Weiers1}). But this may not be always possible if  $\Omega$ is disconnected.

\begin{rem}\label{rek}
Theorem \ref{lema} is a Fritz John (FJ) type necessary optimality condition. In the case where $\lambda_0=0$, no information on the objective functions  can be derived from the necessary optimality condition and it becomes useless. Thus, the case where $\lambda_0=1$ is desirable. It follows from Theorem \ref{lema} that if there is no nonzero abnormal multiplier, i.e., the following implication holds:
\begin{eqnarray*}
&&\left\{\begin{array}{l}
(p(t_0),-p(t_1))\in {\cal N}_E^L(x^*(t_0),x^*(t_1)),\\[3pt]
(\dot{p}(t),0)\in {\rm co}\,\big\{(w,0): (w,0)\in \partial^L\langle -p(t),\phi(t, \cdot,\cdot) \rangle (x^*(t),u^*(t))\\[3pt]
\qquad\qquad\qquad +{\cal N}^L_{S(t)}(x^*(t),u^*(t))\big\}\quad a.e.\ t\in[t_0,t_1],\\[3pt]
(x^{*}(t),u)\in S(t),\ \|u-u^*(t))\|< R(t)\\[3pt]
\qquad\qquad\Longrightarrow\langle p(t),\phi(t,x^*(t),u) \rangle \leq \langle p(t), \phi (t,x^*(t),u^*(t)) \rangle\quad a.e.\ t\in[t_0,t_1],\end{array}\right.\\
&&\Longrightarrow p(t)=0\ {\rm for\ some}\ t\in [t_0,t_1],
\end{eqnarray*}
then the conclusions of Theorem \ref{lema} hold with $\lambda_0=1$.
Such a condition is automatically satisfied in the case of free initial or final point, that is, $E=E_0\times \Re^n$ or $E=\Re^n\times E_1$ with closed subsets $E_0, E_1$ in $\Re^n$.  Supposing $\lambda_0=0$, the result (b) in Theorem \ref{lema} yields that $p(t_1)=0$ or $p(t_0)=0$, respectively, which contradicts the result (a) of this theorem. Throughout the paper, all the derived necessary optimality conditions are FJ type conditions. The desired case where $\lambda_0=1$ can be  obtained provided that there is no nonzero abnormal multiplier, which is always true if the initial or final point is free.
\end{rem}

\section{Necessary optimality conditions for OCPEC}\label{section3}

In this section, we develop  necessary optimality conditions for the OCPEC under the following basic hypothesises.

\begin{assu}[Basic assumption]\label{basic}
$F:[t_0,t_1]\times \Re^n\times \Re^m\to \Re$ and
$\phi:[t_0,t_1]\times \Re^{n}\times \Re^{m}\to \Re^n$ are ${\cal L}\times {\cal B}$ measurable,
$g:[t_0,t_1]\times \Re^{n}\times \Re^m\to \Re^{l_1}, h:[t_0,t_1]\times \Re^{n}\times \Re^m\to \Re^{l_2}$, and $G,H:[t_0,t_1]\times \Re^{n}\times \Re^m\to \Re^l$ are ${\cal L}$ measurable in variable $t$ and strict differentiable in variable $(x,u)$, $U:[t_0,t_1]\rightrightarrows \Re^m$ is a ${\cal L}$ measurable multifunction,
$f:\Re^n\times \Re^n\to \Re$ is locally Lipschitz continuous,  and $E$ is a closed subset in $\Re^n\times \Re^n$.
\end{assu}

In fact, we can easily extend our results to the case where the mappings $g(\cdot),h(\cdot), G(\cdot),H(\cdot)$ are only  Lipschitz continuous in variable $(x,u)$
and strictly differentiable at $(x^*(t),u^*(t))$. But for simplicity of exposition, we assume that they are strictly differentiable in variable $(x,u)$ as in Assumption \ref{basic}.

Given an admissible pair $(x(\cdot),u(\cdot))$ and a point $t\in [t_0,t_1]$, we define the index sets:
\begin{eqnarray*}
&&I^-_t(x,u):=\{i: g_i(t,x(t),u(t))<0\},\\
&&I^0_t(x,u):=\{i: g_i(t,x(t),u(t))=0\},\\
&&\cI^{+0}_t(x,u):=\{i: G_i(t,x(t),u(t))>0=H_i(t,x(t),u(t))\},\\
&&\cI^{00}_t(x,u):=\{i: G_i(t,x(t),u(t))=0=H_i(t,x(t),u(t))\},\\
&&\cI^{0+}_t(x,u):=\{i: G_i(t,x(t),u(t))=0<H_i(t,x(t),u(t))\}.
\end{eqnarray*}
Moreover, for any $(\lambda,\upsilon,\mu,\nu)\in \Re^{l_1}\times \Re^{l_2} \times \Re^{l} \times \Re^l$, we denote
\[
\Psi(t,x,u;\lambda,\upsilon,\mu,\nu):=g(t,x,u)^\top\lambda+h(t,x,u)^\top\upsilon-G(t,x,u)^\top\mu-H(t,x,u)^\top\nu.
\]


\begin{thm}\label{thmw}
Let $(x^*(\cdot),u^*(\cdot))$ be a  local minimizer of radius $R(\cdot)$ for the OCPEC and  Assumption \ref{basic} hold. Suppose that Assumption \ref{ass1} with $S(t)$ defined in  (\ref{cpc1}) is also satisfied.  If for almost every $t\in[t_0,t_1]$, the local error bound condition for the system representing $S(t)$ as in (\ref{cpc1})
holds at $(x^*(t),u^*(t))$, then there exist a number $\lambda_0\in\{0,1\}$, an arc $p(\cdot)$, and measurable functions $\lambda^g:\Re\to \Re^{l_1},\lambda^h:\Re\to \Re^{l_2},\lambda^G:\Re\to \Re^{l},\lambda^H:\Re\to \Re^{l}$ such that the following conditions hold:
\begin{itemize}
\item[{\rm (i)}] the nontriviality condition $(\lambda_0,p(t))\neq0\quad \forall t\in [t_0,t_1]$;
\item[{\rm(ii)}] the transversality condition
$$(p(t_0),-p(t_1))\in \lambda_0\partial^L f(x^*(t_0),x^*(t_1))+{\cal N}_E^L(x^*(t_0),x^*(t_1));$$
\item[{\rm(iii)}] the Euler adjoint inclusion: For almost every $t\in[t_0,t_1]$,
\begin{eqnarray*}\label{w-euler}
\begin{array}{l}
(\dot{p}(t),0)\in \partial^C\big\{\langle -p(t),\phi(t,\cdot,\cdot)\rangle+\lambda_0 F(t,\cdot,\cdot)\big\}(x^*(t),u^*(t))
+\{0\}\times {\cal N}^C_{U(t)}(u^*(t))
\\[5pt]
\qquad \qquad\ \;+ \nabla_{x,u}\Psi(t,x^*(t),u^*(t);\lambda^g(t),\lambda^h(t),\lambda^G(t),\lambda^H(t)),\\[4pt]
 \lambda^g(t)\geq 0,\ \lambda_i^g(t)=0\ \forall i\in I^-_t(x^*,u^*),\\[4pt]
 \lambda^G_i(t)=0\  \forall i\in \cI^{+0}_t(x^*,u^*),\ \lambda^H_i(t)=0\  \forall i\in \cI^{0+}_t(x^*,u^*);
\end{array}
\end{eqnarray*}

\item[{\rm(iv)}] the Weierstrass condition for radius $R(\cdot)$: For almost every $t\in[t_0,t_1]$,
\begin{equation*}
\begin{array}{l}
(x^{*}(t),u)\in S(t), \ \|u-u^*(t)\|< R(t)\quad \Longrightarrow\\[4pt]
\langle p(t),\phi(t,x^*(t),u) \rangle -\lambda_0F(t,x^*(t),u) \leq \langle p(t), \phi (t,x^*(t),u^*(t)) \rangle -\lambda_0 F(t,x^*(t),u^*(t)).
\end{array}
\end{equation*}
\end{itemize}
\end{thm}
\begin{proof}
For simplicity in the proof, we omit the equality and inequality constraints (\ref{cons}),
and the control constraint (\ref{control}) since we have checked that all the formulas of the proof have the corresponding counterparts when using $S(t)$ defined in (\ref{cpc1}) instead of (\ref{cpc}). 
Then  $(x^*(\cdot),u^*(\cdot))$ is a  local minimizer of radius $R(\cdot)$ for $(P_s) $ with $S(t)$ defined  as follows:
\begin{eqnarray}\label{cpc}
S(t)=\left\{(x,u):\big(G(t,x,u), H(t,x,u)\big)\in {\cal C}^l\right\}.
\end{eqnarray}
By virtue of Theorem \ref{lema}, we can easily get the results (i), (ii), and (iv) in this theorem. It now  suffices to show the result (iii) by Theorem \ref{lema}(c).
Since the local error bound condition holds at $(x^*(t),u^*(t))$ and the functions $G(t,\cdot, \cdot), H(t,\cdot, \cdot)$ are strictly differentiable, it  follows from Proposition \ref{normalcone} and \cite[Proposition 3.4]{I-O} that
\begin{eqnarray}\label{m-calcu}
\lefteqn{{\cal N}_{S(t)}^L(x^*(t),u^*(t)) } \nonumber \\
&& \subseteq \big\{ -\nabla_{x,u} G(t,x^*(t),u^*(t))\beta-\nabla_{x,u} H(t,x^*(t),u^*(t))\gamma:(\beta,\gamma)\in{\cal M}^*(t)\big\},
\end{eqnarray}
where
\begin{equation}\label{m-set}
{\cal M}^*(t):=\left\{(\beta,\gamma):\begin{array}{l}\beta_i=0\quad {\rm if}\ i\in \cI^{+0}_t(x^*,u^*)\\[2pt] \gamma_i=0\quad {\rm if}\ i\in\cI^{0+}_t(x^*,u^*)\\[2pt] \beta_i>0,\gamma_i>0 {\ \rm or \ } \beta_i\gamma_i=0 \quad{\rm if}\ i\in \cI^{00}_t(x^*,u^*)\end{array}\right\}.
\end{equation}
It then follows from (\ref{euler adjoint}) and \eqref{m-calcu} that for almost every $t\in[t_0,t_1]$,
\begin{equation}\label{thmC-4}
\begin{array}{l}
(\dot{p}(t),0)\in \partial^C\big\{\langle -p(t),\phi(t,\cdot, \cdot)\rangle+\lambda_0 F(t,\cdot, \cdot)\big\}(x^*(t),u^*(t))\\[4pt]
\qquad \qquad +{\rm co}\left\{-\nabla_{x,u}G(t,x^*(t),u^*(t))\beta - \nabla_{x,u}H(t,x^*(t),u^*(t))\gamma:(\beta,\gamma)\in  {\cal M}^*(t)\right\}.
\end{array}
\end{equation}
By Carath\'{e}odory's theorem for the convex hull, it then follows from \eqref{thmC-4} that for almost every $t\in[t_0,t_1]$, there exist $\alpha \in \Delta :=\{\alpha\in\Re^{n+m+1}_+: \sum_{j=1}^{n+m+1} \alpha_j=1\}$ and $(\beta^j,\gamma^j)\in{\cal M}^*(t)\ (\forall j=1,\ldots,n+m+1)$ such that
\begin{eqnarray}
\psi(t, \alpha, \beta,\gamma)\in -(\dot{p}(t),0)+\partial^C\big\{\langle -p(t),\phi(t,\cdot,\cdot)\rangle+\lambda_0 F(t,\cdot,\cdot)\big\}(x^*(t),u^*(t)),
\end{eqnarray}
where
\[
\psi(t, \alpha,\beta,\gamma):=\sum_{j=1}^{n+m+1}\alpha_j \left[ \nabla_{x,u}G(t,x^*(t),u^*(t))\beta^j+ \nabla_{x,u}H(t,x^*(t),u^*(t))\gamma^j\right]
\]
is a Carath\'{e}odory mapping since it is  continuous in $(\alpha,\beta,\gamma)$ and  measurable in $t$ by virtue of  \cite[Theorem 14.13]{RW}. By \cite[Theorem 14.56 and Exercise 14.12]{RW}, the multifunction
$$
\partial^C\big\{\langle -p(t),\phi(t,\cdot,\cdot)\rangle+\lambda_0 F(t,\cdot,\cdot)\big\}(x^*(t),u^*(t))
$$
is measurable in $t$.
Hence,
\[
-(\dot{p}(t),0)+\partial^C\big\{\langle -p(t),\phi(t,\cdot,\cdot)\rangle+\lambda_0 F(t,\cdot,\cdot)\big\}(x^*(t),u^*(t))
\]
is measurable in $t$.  Moreover, by \cite[Theorem 14.26]{RW}, the multifunction ${\cal M}^*(t)$ is measurable in $t$.
Thus, it follows from the implicit measurable function theorem \cite[Theorem 14.16]{RW} that there exist
measurable functions $\alpha(\cdot)\in \Delta$  and $(\beta^j(\cdot),\gamma^j(\cdot))\in{\cal M}^*(\cdot)\ (\forall j=1,\ldots,n+m+1)$ such that for almost every $t\in[t_0,t_1]$,
\begin{eqnarray*}
\lefteqn{(\dot{p}(t),0)\in\partial^C\big\{\langle -p(t),\phi(t,\cdot,\cdot)\rangle+\lambda_0 F(t,\cdot,\cdot)\big\}(x^*(t),u^*(t))}
\\[3pt]
&& -\sum_{j=1}^{n+m+1}\alpha_j(t) \nabla_{x,u}G(t,x^*(t),u^*(t))\beta^j(t)- \sum_{j=1}^{n+m+1} \alpha_j(t)\nabla_{x,u}H(t,x^*(t),u^*(t))\gamma^j(t).
\end{eqnarray*}
Let
\[
\lambda^G(\cdot):=\sum_{j=1}^{n+m+1}\alpha_j(\cdot) \beta^j(\cdot),\quad \lambda^H(\cdot):=\sum_{j=1}^{n+m+1} \alpha_j(\cdot)\gamma^j(\cdot),
\]
which are both clearly  measurable in $t$. Moreover, since $(\beta^j(t),\gamma^j(t))\in{\cal M}^*(t)\ (\forall j=1,\ldots,n+m+1)$, it is not hard to see that
\[
\lambda^G_i(t)=0 \quad \forall i\in\cI^{+0}_t(x^*,u^*),\quad \lambda^H_i(t)=0\quad \forall i\in \cI^{0+}_t(x^*,u^*).
\]
Thus, the desired result follows immediately. The proof is complete.
\end{proof}

By virtue of the Weierstrass condition for radius $R(\cdot)$ (Theorem \ref{thmw}(iv)), we have that for almost every $t\in[t_0,t_1]$, $u^*(t)$ is a local minimizer of the following  MPEC:
\begin{eqnarray}\label{mpec}
\min\limits_{u\in U(t)} && -\langle p(t),\phi(t,x^*(t),u)\rangle+\lambda_0 F(t,x^*(t),u)\nonumber\\
{\rm s.t.}&&  g(t,x^*(t),u)\leq 0,\  h(t,x^*(t),u)=0,\\
&&0\leq G(t,x^*(t),u)\ \bot\ H(t,x^*(t),u)\geq 0.\nonumber
\end{eqnarray}
Hence, under some constraint qualifications for MPEC \eqref{mpec}, the popular necessary conditions such as the C-, M-, and S-stationarities may hold at $u^*(t)$; see, e.g., \cite{Scholtes-stationarity,kanzow10,Jane-jmaa2005,Jane-Jin}. This and Theorem \ref{thmw} motivate us to define the following stationarity conditions.

\begin{defi}\label{defi-w}
Let  $(x^*(\cdot),u^*(\cdot))$  be an admissible pair of the OCPEC. We say that the FJ type weak stationarity (W-stationarity) holds at $(x^*(\cdot),u^*(\cdot))$ if there exist a number $\lambda_0\in\{0,1\}$, an arc $p(\cdot)$, and  measurable functions $\lambda^g(\cdot),\lambda^h(\cdot),\lambda^G(\cdot),\lambda^H(\cdot)$ such that Theorem \ref{thmw}(i)-(iv) hold.

We say that the FJ type C-stationarity holds at $(x^*(\cdot),u^*(\cdot))$ if $(x^*(\cdot),u^*(\cdot))$ is W-stationary with arc $p(\cdot)$ and there exist measurable functions $\eta^g(\cdot),\eta^h(\cdot),\eta^G(\cdot),\eta^H(\cdot)$ such that for almost every $t\in[t_0,t_1]$,
\begin{eqnarray}
&&0\in \partial^L_u\big\{-\langle p(t), \phi(t,x^*(t),\cdot)\rangle+\lambda_0 F(t,x^*(t),\cdot)\big\}(u^*(t))\nonumber\\
&&\qquad\quad + \nabla_u\Psi(t,x^*(t),u^*(t);\eta^g(t),\eta^h(t),\eta^G(t),\eta^H(t))+{\cal N}^L_{U(t)}(u^*(t)),\label{w-l}\\[4pt]
&&\eta^g(t)\geq 0,\ \eta_i^g(t)=0\ \forall i\in I^-_t(x^*,u^*),\label{w-ine}\\[4pt]
&& \eta^G_i(t)=0\ \forall i\in \cI^{+0}_t(x^*,u^*),\ \eta^H_i(t)=0\ \forall i\in \cI^{0+}_t(x^*,u^*),\label{w-w}
\end{eqnarray}
and
\begin{eqnarray*}
\eta^G_i(t)\eta^H_i(t)\geq 0 \ \forall i \in \cI^{00}_t(x^*,u^*)\label{w-c}.
\end{eqnarray*}

We say that the FJ type M-stationarity holds at $(x^*(\cdot),u^*(\cdot))$ if $(x^*(\cdot),u^*(\cdot))$ is W-stationary with arc $p(\cdot)$ and there exist measurable functions $\eta^g(\cdot),\eta^h(\cdot),\eta^G(\cdot),\eta^H(\cdot)$ such that for almost every $t\in[t_0,t_1]$, \eqref{w-l}--\eqref{w-w} hold and
\begin{equation}\label{w-m}
\eta^G_i(t)> 0, \eta^H_i(t)>0 \mbox{ or } \eta^G_i(t)\eta^H_i(t) =0 \ \forall i \in \cI^{00}_t(x^*,u^*).
\end{equation}

We say that the FJ type S-stationarity holds at $(x^*(\cdot),u^*(\cdot))$ if $(x^*(\cdot),u^*(\cdot))$ is W-stationary with arc $p(\cdot)$ and there exist measurable functions $\eta^g(\cdot),\eta^h(\cdot),\eta^G(\cdot),\eta^H(\cdot)$ such that for almost every $t\in[t_0,t_1]$, \eqref{w-l}--\eqref{w-w} hold and
\[
\eta^G_i(t)\geq 0, \eta^H_i(t)\geq0 \ \forall i \in \cI^{00}_t(x^*,u^*).
\]

We will refer to them as the W-, C-, M-, and S-stationarities, respectively, if $\lambda_0=1$.
\end{defi}

In  Definition \ref{defi-w}, there are two sets of multipliers.
The ideal situation is when these two sets of multipliers are identical almost everywhere. In the case where the multipliers $\lambda^g(\cdot)$, $\lambda^h(\cdot)$, $\lambda^G(\cdot)$, $\lambda^H(\cdot)$ and
$\eta^g(\cdot),\eta^h(\cdot),\eta^G(\cdot),\eta^H(\cdot)$ can be chosen as the same almost everywhere, $(x^*(\cdot),u^*(\cdot))$ being  C-, M-, S-stationarities becomes that $(x^*(\cdot),u^*(\cdot))$ is W-stationary with multipliers
satisfying the following extra sign conditions:
\begin{eqnarray*}
&&  \lambda^G_i(t)\lambda^H_i(t)\geq 0 \quad \forall i \in \cI^{00}_t(x^*,u^*)\quad {\rm a.e.}\ t\in[t_0,t_1];\\
&& \lambda^G_i(t)> 0, \lambda^H_i(t)>0 \mbox{ or } \lambda^G_i(t)\lambda^H_i(t) =0 \quad \forall i \in \cI^{00}_t(x^*,u^*)\quad {\rm a.e.}\ t\in[t_0,t_1];\\
&&\lambda^G_i(t)\geq  0, \lambda^H_i(t)\geq 0  \quad \forall i \in \cI^{00}_t(x^*,u^*)\quad {\rm a.e.}\ t\in[t_0,t_1],
\end{eqnarray*} respectively.
Although we hope that these two sets of multipliers can be chosen as the same almost everywhere, the following example shows that it is not always possible.

\begin{ex}\rm
Consider the problem
\begin{eqnarray*}
\min && x(t_1)\\
{\rm s.t.} && \dot{x}(t)= u(t)\quad {\rm a.e.}\ t\in[t_0,t_1], \\
           &&  0\leq -u(t) \ \bot \ x(t)-u^2(t) \geq0\quad {\rm a.e.}\  t\in[t_0,t_1], \\
           && x(t_0)\leq0,
\end{eqnarray*}
where $x,u:\Re\to\Re$. Since $x(\cdot)$ is absolutely continuous and $x(t)\geq0$ for almost every $t\in[t_0,t_1]$, we must have $x(t)\geq0$ for every $t\in[t_0,t_1]$. Then it is easy to see that $(x^*(\cdot),u^*(\cdot))\equiv(0,0)$ is a minimizer of the above problem. Moreover, it is not hard to verify that for the system $\Omega:=\{u: F(u)\in \cC^1\}$ with $F(u):=(-u, x^*(t)-u^2)^T$ and $\cC^1$ defined as in (\ref{defi C}),
\[
\Re={\cal T}_\Omega(u^*(t))^o\subseteq \nabla F(u^*(t)){\cal N}_{\cC^1}(F(u^*(t)))=\Re,
\]
where ${\cal T}_\Omega(u^*(t))^o$ stands for the polar of the tangent cone to $\Omega$ at $u^*(t)$. It has been shown in \cite[Theorem 3.2]{guolin13} that this condition ${\cal T}_\Omega(u^*(t))^o\subseteq \nabla F(u^*(t)){\cal N}_{\cC^1}(F(u^*(t)))$  is a constraint qualification for M-stationarity at $u^*(t)$. Thus, for almost every $t\in [t_0,t_1]$, $u^*(t)=0$ is M-stationary to the problem
$$\min_u\ - p(t)u\quad {\rm s.t.} \quad 0\leq -u \ \bot \ x^*(t)-u^2 \geq 0.$$
By solving the M-stationarity condition at $(x^*(\cdot),u^*(\cdot))$, we have
\begin{eqnarray}
&& p(t_0)\geq 0,\ p(t_1)= -1,\label{ex1}\\
&& \dot{p}(t)=-\lambda^H(t),\ p(t)=\lambda^G(t) \quad {\rm a.e.}\ t\in[t_0,t_1],\label{ex2}\\
&& p(t) = \eta^G(t),\  \eta^G(t)>0, \eta^H(t)>0\ {\rm or}\ \eta^G(t)\eta^H(t)=0\quad {\rm a.e.}\ t\in[t_0,t_1].\label{ex4}
\end{eqnarray}
Since $p(\cdot)$ is absolutely continuous, by virtue of \eqref{ex1}, there must exist an interval $[t',t'']\subseteq[t_0,t_1]$ with $t'<t''$ such that
\[
p(t)<0,\ \dot{p}(t)<0\quad \forall t\in [t',t''].
\]
This together with \eqref{ex2}--\eqref{ex4} implies
\begin{eqnarray*}
&&\lambda^G(t)<0, \lambda^H(t)>0 \quad {\rm a.e.}\ t\in[t',t''],\\
&&\eta^G(t)<0, \eta^H(t)=0 \quad {\rm a.e.}\ t\in[t',t''],
\end{eqnarray*}
which shows that $\lambda^H(t) \not =\eta^H(t)$ for almost every $t\in[t',t'']$.
\end{ex}

We now show that the FJ type M-stationarity for the OCPEC in Definition \ref{defi-w} is necessary for optimality under certain constraint qualifications. Note that problem \eqref{mpec} is an MPEC with respect to variable $u$. In the following theorem, we will assume that some MPEC constraint qualifications for M-stationarity, which are qualifications to derive M-stationarity  for optimality, are satisfied.
The reader is referred to  \cite{kanzow10,Jane-jmaa2005,Jane-Jin,guolin13} and the references within for MPEC constraint qualifications for M-stationarity.

\begin{thm}\label{thmm}
Let $(x^*(\cdot),u^*(\cdot))$ be a  local minimizer of radius $R(\cdot)$ for the OCPEC and  Assumption \ref{basic} hold. Suppose that Assumption \ref{ass1} with $S(t)$ defined in  (\ref{cpc1}) is also satisfied. Then the FJ type M-stationarity holds at $(x^*(\cdot),u^*(\cdot))$  if for almost every $t\in[t_0,t_1]$, one of the following conditions holds:
\begin{itemize}
\item[\rm (a)] The local error bound condition for the system representing $S(t)$ as in (\ref{cpc1})
holds at $(x^*(t),u^*(t))$ and an MPEC constraint qualification for M-stationarity  holds at $u^*(t)$  for problem (\ref{mpec});

\item[\rm (b)] The MPEC linear condition holds for $S(t)$ defined in (\ref{cpc1}), i.e., functions $g(t,\cdot, \cdot)$, $h(t,\cdot, \cdot)$, $G(t,\cdot, \cdot)$, $H(t,\cdot, \cdot)$  are affine in $(x,u)$ and $U(t)$ is a union of finitely many polyhedral sets;

\item[\rm (c)] The MPEC quasi-normality holds at $u^*(t)$  for problem (\ref{mpec}), i.e., there is no nonzero multiplier $(\lambda,\upsilon,\mu,\nu)$ such that
\begin{itemize}
\item[--] $0\in \nabla_u \Psi(t,x^*(t),u^*(t);\lambda,\upsilon,\mu,\nu)+\cN^L_{U(t)}(u^*(t))$,
\item[--] $\lambda\geq0$,\ $\lambda_i=0\ \forall i\in I_t^-(x^*,u^*)$,\ $\mu_i=0\ \forall i\in \cI^{+0}_t(x^*,u^*),\ \nu_i=0\ \forall i\in \cI^{0+}_t(x^*,u^*)$,\ $\mu_i>0,\nu_i>0 \mbox{\ \rm or\ } \mu_i\nu_i=0\ \forall i\in \cI^{00}_t(x^*,u^*)$,
\item[--]  there exists a sequence $\{u^k\}\subseteq U(t)$ converging to $u^*(t)$ such that for
each $k$,
\begin{eqnarray*}
&& \lambda_i>0 \ \Longrightarrow \ g_i(t,x^*(t),u^k)>0,\quad \upsilon_i\neq0 \ \Longrightarrow \ \upsilon_ih_i(t,x^*(t),u^k)>0,\\
&&\mu_i\neq0  \ \Longrightarrow \ \mu_iG_i(t,x^*(t),u^k)<0, \quad \nu_i\neq0  \ \Longrightarrow \  \nu_iH_i(t,x^*(t),u^k)<0.
\end{eqnarray*}
\end{itemize}

\end{itemize}
\end{thm}
\begin{proof}
First we observe that for almost every $t\in[t_0,t_1]$, the local error bound condition for the system representing $S(t)$ as in (\ref{cpc1}) holds at $(x^*(t),u^*(t))$ under either the MPEC linear condition in condition (b) or the MPEC quasi-normality in condition (c).  Thus, it follows from Theorem \ref{thmw} that $(x^*(\cdot),u^*(\cdot))$ is W-stationary under any one condition. Moreover,  conditions (a), (b), and (c) can all imply that for almost every $t\in[t_0,t_1]$,  there exist $\eta^g(t),\eta^h(t),\eta^G(t),\eta^H(t)$ such that \eqref{w-l}--\eqref{w-m} hold (\cite[Theorem 2.2]{Jane-jmaa2005} and \cite[Theorem 3.3]{kanzow10}). By the implicit measurable function theorem (see, e.g., \cite[Theorem 14.16]{RW}), the functions $\eta^g(\cdot),\eta^h(\cdot),\eta^G(\cdot),\eta^H(\cdot)$ can be chosen measurably. The proof is complete.
\end{proof}


We next derive the FJ type S-stationarity under the MPEC LICQ. It should be noted that the MPEC LICQ is generic and hence not a stringent assumption by \cite{Scholtes-MPECLICQ}.

\begin{thm}\label{thmS}
Let $(x^*(\cdot),u^*(\cdot))$ be a  local minimizer of radius $R(\cdot)$ for the OCPEC and Assumption \ref{basic} hold. Suppose that Assumption \ref{ass1} with $S(t)$ defined in  (\ref{cpc1}) is also satisfied.  Assume further that for almost every $t\in[t_0,t_1]$, $U(t)=\Re^m$ and the functions $F(t,\cdot, \cdot)$, $\phi(t,\cdot,\cdot)$ are strictly differentiable at $(x^*(t),u^*(t))$. If for almost every $t\in[t_0,t_1]$, the MPEC LICQ holds at $u^*(t)$ for problem (\ref{mpec}), i.e., the family of gradients
\begin{eqnarray*}
\lefteqn{\left\{\nabla_u g_i(t,x^*(t),u^*(t)):i\in I^0_t(x^*,u^*)\right\}\cup \left\{\nabla_u h_i(t,x^*(t),u^*(t)):i=1,\ldots,l_2\right\}}\\
&&\cup\left\{\nabla_u G_i(t,x^*(t),u^*(t)):i\in \cI_t^{0\bullet}(x^*,u^*)\right\}\cup \left\{\nabla_u H_i(t,x^*(t),u^*(t)): i\in \cI_t^{\bullet0}(x^*,u^*)\right\}
\end{eqnarray*}
where $\cI_t^{0\bullet}(x^*,u^*):=\cI_t^{0+}(x^*,u^*)\cup \cI_t^{00}(x^*,u^*)$ and $\cI_t^{\bullet0}(x^*,u^*):=\cI_t^{+0}(x^*,u^*)\cup \cI_t^{00}(x^*,u^*)$, is linearly independent, then the FJ type S-stationarity holds at $(x^*(\cdot),u^*(\cdot))$. Moreover, the multipliers $\eta^g(\cdot),\eta^h(\cdot),\eta^G(\cdot), \eta^H(\cdot)$ can be taken as equal to $\lambda^g(\cdot),\lambda^h(\cdot),\lambda^G(\cdot), \lambda^H(\cdot)$ almost everywhere. That is, there exist a number $\lambda_0\in\{0,1\}$, an arc $p(\cdot)$, and measurable functions $\lambda^g(\cdot),\lambda^h(\cdot),\lambda^G(\cdot), \lambda^H(\cdot)$ such that $(x^*(\cdot),u^*(\cdot))$ is W-stationary and for almost every $t\in[t_0,t_1]$, the following extra sign condition holds:
\[
\lambda^G_i(t)\geq  0, \lambda^H_i(t)\geq 0  \quad \forall i \in \cI^{00}_t(x^*,u^*).
\]
\end{thm}
\begin{proof}
Under the MPEC LICQ assumption, by Proposition \ref{error bound}, it follows that for almost every $t\in[t_0,t_1]$, the local error bound condition for the system representing $S(t)$ as in (\ref{cpc1}) holds at $(x^*(t), u^*(t))$.  Thus, it follows from Theorem \ref{thmw} that $(x^*(\cdot),u^*(\cdot))$ is W-stationary. Moreover, for almost every $t\in[t_0,t_1]$, since the MPEC LICQ holds at $u^*(t)$, it then follows from \cite[Theorem 2]{Scholtes-stationarity} that there exist $\eta^g(t),\eta^h(t),\eta^G(t),\eta^H(t)$ such that \eqref{w-l}--\eqref{w-w} hold and
\[
\eta^G_i(t)\geq 0, \eta^H_i(t)\geq0 \quad i \in \cI^{00}_t(x^*,u^*).
\]
By the implicit measurable function theorem (see, e.g., \cite[Theorem 14.16]{RW}), the functions $\eta^g(\cdot),\eta^h(\cdot),\eta^G(\cdot),\eta^H(\cdot)$ can be chosen measurably. Thus, the first part of the theorem is derived.
Moreover, by the MPEC-LICQ assumption, it is not hard to see that $\lambda^g(t)=\eta^g(t),\lambda^g(t)=\eta^g(t),\lambda^G(t)=\eta^G(t), \lambda^H(t)=\eta^H(t)$ for almost every $t\in [t_0,t_1]$.
Therefore, the second part of the theorem follows immediately. The proof is complete.
\end{proof}

For problem $(P_s)$, if $S(t)=\Re^n\times U(t)$ for almost every $t\in[t_0,t_1]$ (which corresponds to the case of standard optimal control problem without mixed constraints), then the bounded slope condition \eqref{bsc} holds automatically  for almost every $t\in[t_0,t_1]$ since in this case, \eqref{bsc} becomes
\[
(x,u) \in S_*^{\epsilon,R}(t),\ \beta\in {\cal N}^P_{U(t)}(u)\Longrightarrow  k_S(t)\|\beta\|\geq0,
\]
which holds trivially if $k_S(t)\geq0$. If there exists a closed subset $X(t')\subseteq \Re^n$ and $\bar{x}^*(t')\in {\rm bd}\,X(t')$
satisfying $S(t')=X(t')\times \Re^m$ and
\[
{\rm dist}_{{\rm bd}\,X(t')}(x^*(t'))=\|x^*(t')-\bar{x}^*(t')\|\leq\epsilon,
\]
then \eqref{bsc} at time $t'$ never hold since there exists $0\neq\alpha\in {\cal N}^P_{X(t')}(\bar{x}^*(t'))$ by \cite[Exercise 6.19]{RW}. If the set of such a point $t'\in[t_0,t_1]$ is not of  measure zero, then the bounded slope condition in Assumption \ref{ass1} does not hold. As a consequence, the bounded slope condition can hardly hold for the case of the pure state constraint $S(t)=X(t)\times \Re^m$. Generally speaking, the bounded slope condition is a strong condition and  is also hard to verify. In the rest of this section, we will investigate sufficient conditions for the bounded slope condition to hold in our problem setting.

\begin{prop}\label{prop-bsc}
Assume that the local error bound condition for the system representing $S(t)$ as in (\ref{cpc1})
holds at every $(x,u)\in S_*^{\epsilon,R}(t)$ and
\begin{eqnarray}\label{bsc-ms}
&&\left\{\begin{array}{l}
(x,u)\in S_*^{\epsilon,R}(t),\ \zeta\in \cN^L_{U(t)}(u),\\[2pt]
\lambda\geq0,\ \lambda_i=0\ \forall i\in I_t^-(x,u),\\[2pt]
\mu_i=0\ \forall i\in \cI^{+0}_t(x,u),\ \nu_i=0\ \forall i\in \cI^{0+}_t(x,u),\\[2pt]
 \mu_i>0, \nu_i> 0\ {\rm or}\  \mu_i\nu_i=0\ \forall i\in \cI^{00}_t(x,u),
\end{array}\right.\nonumber\\[4pt]
&& \Longrightarrow \|\nabla_x \Psi(t,x,u;\lambda,\upsilon,\mu,\nu)\|\leq k_S(t)\|\nabla_u\Psi(t,x,u;\lambda,\upsilon,\mu,\nu)+\zeta\|.
\end{eqnarray}
Then the bounded slope condition (\ref{bsc}) holds.
\end{prop}

\begin{proof}
Let $(x,u)\in S_*^{\epsilon,R}(t)$ and $(\alpha,\beta)\in {\cal N}^L_{S(t)}(x,u)$. Since the local error bound condition holds at $(x,u)$, it then follows from \cite[Proposition 3.4]{I-O} that
\begin{eqnarray*}
(\alpha,\beta)\in\left\{\nabla_{x,u} \Psi(t,x,u;\lambda,\upsilon,\mu,\nu):
\begin{array}{l}
\lambda\in \cN^L_{\Re_-^{l_1}}(g(t,x,u)), \upsilon \in \Re^{l_2}\\
(\mu,\nu)\in \cN^L_{\cC^l}(G(t,x,u),H(t,x,u))
\end{array}
\right\}
 +\{0\}\times {\cal N}^L_{U(t)}(u).
\end{eqnarray*}
Then, by Proposition \ref{normalcone},  there exist $\lambda,\upsilon,\mu,\nu$ such that
\begin{eqnarray*}
\begin{array}{lr}
(\alpha,\beta)\in \nabla_{x,u} \Psi(t,x,u;\lambda,\upsilon,\mu,\nu)+\{0\}\times {\cal N}^L_{U(t)}(u),\ \lambda\geq0,\ \lambda_i=0\ \forall i\in I_t^-(x,u),\\[4pt]
\mu_i=0\ \forall i\in \cI^{+0}_t(x,u),\ \nu_i=0\ \forall i\in \cI^{0+}_t(x,u), \ \mu_i>0, \nu_i> 0\ {\rm or}\  \mu_i\nu_i=0\ \forall i\in \cI^{00}_t(x,u).
\end{array}
\end{eqnarray*}
It then follows that there exists $\zeta\in \cN^L_{U(t)}(u)$ such that
\[
\alpha=\nabla_x \Psi(t,x,u;\lambda,\upsilon,\mu,\nu),\quad \beta=\nabla_u\Psi(t,x,u;\lambda,\upsilon,\mu,\nu)+\zeta.
\]
Thus, by condition (\ref{bsc-ms}), we have $\|\alpha\|\leq k_S(t)\|\beta\|$. The proof is complete.
\end{proof}

%

A sufficient condition for condition (\ref{bsc-ms}) to hold is the following stronger condition that is similar to the $M_*^{\epsilon,R}$ condition given in \cite{C-P}: There exists a measurable function $\kappa(\cdot)$ such that for almost every $t\in[t_0,t_1]$,
\begin{eqnarray}\label{bsc2}
&&\left\{\begin{array}{l}
(x,u)\in S_*^{\epsilon,R}(t),\ \zeta\in \cN^L_{U(t)}(u),\\[2pt]
\lambda\geq0,\ \lambda_i=0\ \forall i\in I_t^-(x,u),\\[2pt]
\mu_i=0\ \forall i\in \cI^{+0}_t(x,u),\ \nu_i=0\ \forall i\in \cI^{0+}_t(x,u),\\[2pt]
\mu_i>0, \nu_i> 0\ {\rm or}\  \mu_i\nu_i=0\ \forall i\in \cI^{00}_t(x,u),
\end{array}\right.\nonumber\\[4pt]
&& \Longrightarrow \|(\lambda,\upsilon,\mu,\nu)\|\leq \kappa(t)\|\nabla_u\Psi(t,x,u;\lambda,\upsilon,\mu,\nu)+\zeta\|.
\end{eqnarray}

\begin{assu}\label{ass2}
There exist measurable functions $k_x^g(\cdot),k_x^h(\cdot),k_x^G(\cdot),k_x^H(\cdot)$ such that for almost every $t\in[t_0,t_1]$,
\begin{eqnarray*}
&&\|g(t,x_1,u)-g(t,x_2,u)\|\leq k_x^g(t)\|x_1-x_2\|\quad \forall (x_1,u),(x_2,u)\in S_*^{\epsilon,R}(t),\\
&&\|h(t,x_1,u)-h(t,x_2,u)\|\leq k_x^h(t)\|x_1-x_2\|\quad \forall (x_1,u),(x_2,u)\in S_*^{\epsilon,R}(t),\\
&&\|G(t,x_1,u)-G(t,x_2,u)\|\leq k_x^G(t)\|x_1-x_2\|\quad \forall (x_1,u),(x_2,u)\in S_*^{\epsilon,R}(t),\\
&&\|H(t,x_1,u)-H(t,x_2,u)\|\leq k_x^H(t)\|x_1-x_2\|\quad \forall (x_1,u),(x_2,u)\in S_*^{\epsilon,R}(t).
\end{eqnarray*}
\end{assu}



\begin{prop} \label{prop-strongcon}
Let Assumption \ref{ass2} and condition (\ref{bsc2}) hold.  Then the local error bound condition for the system representing $S(t)$ as in (\ref{cpc1}) holds at every $(x,u)\in S_*^{\epsilon,R}(t)$ and the bounded slope condition (\ref{bsc}) holds with $k_S(t)=\kappa (t)(k_x^g(t)+k_x^h(t)+k_x^G(t)+k_x^H(t))$.
\end{prop}
\begin{proof}
Let $(x,u)\in S_*^{\epsilon,R}(t)$. It is not hard to verify that condition (\ref{bsc2}) implies that
\begin{eqnarray*}
&&\left\{\begin{array}{l}
 0\in \nabla_u \Psi(t,x,u;\lambda,\upsilon,\mu,\nu)+\cN^L_{U(t)}(u),\\[2pt]
\lambda\geq0,\ \lambda_i=0\ \forall i\in I_t^-(x,u),\\[2pt]
\mu_i=0\ \forall i\in \cI^{+0}_t(x,u),\ \nu_i=0\ \forall i\in \cI^{0+}_t(x,u),\\[2pt]
\mu_i>0,\ \nu_i> 0\ {\rm or}\  \mu_i\nu_i=0\ \forall i\in \cI^{00}_t(x,u),
\end{array}\right.\nonumber\\[4pt]
&& \Longrightarrow (\lambda,\upsilon,\mu,\nu)=0.
\end{eqnarray*}
This indicates that the MPEC quasi-normality holds at $(x,u)$ and then by Proposition \ref{error bound}, the local error bound condition for the system representing $S(t)$ as in \eqref{cpc1} holds at $(x,u)$. In the same way as in \cite[Proposition 4.2]{C-P}, we can have that condition (\ref{bsc-ms}) holds with $k_S(t):=\kappa (t)(k_x^g(t)+k_x^h(t)+k_x^G(t)+k_x^H(t))$. Consequently, the bounded slope condition follows from Proposition \ref{prop-bsc} immediately. The proof is complete.
\end{proof}

%
%

In general, it is not easy to guarantee the integrability of the measurable function $k_S(\cdot)$ in the bounded slope condition (\ref{bsc}). We next consider a special case where the mappings $g(\cdot),h(\cdot),G(\cdot),H(\cdot)$, $U(\cdot)$ are all autonomous (i.e., independent of $t$). In this case, we will give some sufficient conditions to ensure that the function $k_S(\cdot)$ is a positive constant function which is clearly integrable. We denote $U(t)\equiv U$, $S(t)\equiv S$, ${\mathbb S}(x):=\{u: (x,u)\in S\}$, and
$$
C^{\epsilon,R}_*:={\rm cl}\,\{(t,x,u)\in[t_0,t_1]\times \Re^n\times \Re^m: (x,u)\in S^{\epsilon,R}_*(t)\}.
$$
Note that  $C^{\epsilon,R}_*$ may be unbounded since $u^*(\cdot)$ may be unbounded on $[t_0,t_1]$.

\begin{prop}\label{prop-auto}
Let all the mappings $g(\cdot),h(\cdot),G(\cdot),H(\cdot),U(\cdot)$ be autonomous. Assume that $C^{\epsilon,R}_*$ is compact for some $\epsilon>0$ and $D^*{\mathbb S}(x,u)(0)=\{0\}$ for every $(x,u)$ such that $(t,x,u)\in C^{\epsilon,R}_*$. Then there exists certain positive constant $\pi$ such that for every $t\in[t_0,t_1]$, the bounded slope condition (\ref{bsc}) holds with $k_S(t)=\pi$. A sufficient condition for $D^*{\mathbb S}(x,u)(0)=\{0\}$ to hold is the local error bound condition for the system representing $S(t)$ as in (\ref{cpc1}) at $(x,u)$ and the implication
\begin{eqnarray}\label{wbcq}
&&\left\{\begin{array}{l}
0\in \nabla_u \Psi(t,x,u;\lambda,\upsilon,\mu,\nu)+\cN^L_U(u),\\[2pt]
\lambda\geq0,\ \lambda_i=0\ \forall i\in I_t^-(x,u),\\[2pt]
\mu_i=0\ \forall i\in \cI^{+0}_t(x,u),\ \nu_i=0\ \forall i\in \cI^{0+}_t(x,u),\\[2pt]
\mu_i>0,\ \nu_i> 0\ {\rm or}\  \mu_i\nu_i=0\ \forall  i\in \cI^{00}_t(x,u),
\end{array}\right.\nonumber \\[4pt]
&& \Longrightarrow \nabla_x \Psi(t,x,u;\lambda,\upsilon,\mu,\nu)=0.
\end{eqnarray}
\end{prop}
\begin{proof}
We prove the first part of this result by contradiction. Assume to the contrary that for every $k$, there exist  $t_k\in[t_0,t_1], (x^k,u^k)\in S^{\epsilon,R}_*(t_k)$, and $(\alpha^k,\beta^k)\in {\cal N}^L_{S}(x^k,u^k)$ such that $\|\alpha^k\|> k\|\beta^k\|.$ Without loss of generality, we assume that $\|\alpha^k\|=1$ and $\alpha^k\to \alpha$ with $\|\alpha\|=1$. Since $\|\alpha^k\|> k\|\beta^k\|\ \forall k$, it follows that $\beta^k\to 0$. Since $C^{\epsilon,R}_*$ is compact, we may assume that $(t_k,x^k,u^k)\to (t,x,u)\in C^{\epsilon,R}_*$. Since the limiting normal cone mapping ${\cal N}^L_{S}(\cdot)$ is closed, we can have $(\alpha,0)\in {\cal N}^L_{S}(x,u)$ that means $\alpha\in D^*{\mathbb S}(x,u)(0)$ by the definition of coderivative.  The assumption $D^*{\mathbb S}(x,u)(0)=\{0\}$ gives a contradiction with the relation $\|\alpha\|=1$. The proof for the first part of this theorem is complete.

Next we show the second part of this theorem.  For any $\alpha\in D^*{\mathbb S}(x,u)(0)$, by the definition of coderivative, we have $(\alpha,0)\in \cN_S^L(x,u)$. Then, as in the proof of Proposition \ref{prop-bsc}, there exist $\lambda,\upsilon,\mu,\nu$ such that
\begin{eqnarray*}
\begin{array}{lr}
(\alpha,0)\in \nabla_{x,u} \Psi(t,x,u;\lambda,\upsilon,\mu,\nu)+\{0\}\times {\cal N}^L_{U}(u),\
\lambda\geq0,\ \lambda_i=0\ \forall i\in I_t^-(x,u),\\[4pt]
\mu_i=0\ \forall i\in \cI^{+0}_t(x,u),\ \nu_i=0\ \forall i\in \cI^{0+}_t(x,u), \ \mu_i>0, \nu_i> 0\ {\rm or}\  \mu_i\nu_i=0\ \forall i\in \cI^{00}_t(x,u).
\end{array}
\end{eqnarray*}
It then follows that
\[
\alpha=\nabla_x \Psi(t,x,u;\lambda,\upsilon,\mu,\nu),\quad 0\in\nabla_u\Psi(t,x,u;\lambda,\upsilon,\mu,\nu) +\cN^L_{U}(u),
\]
which together with condition (\ref{wbcq}) implies that $\alpha=0$. The proof for the second part of the theorem is complete.
\end{proof}

%

The following condition that is stronger than condition \eqref{wbcq} is similar to the so-called MFC proposed in \cite{C-P}:
\begin{eqnarray}\label{wbcqnew}
&&\left\{\begin{array}{l}
0\in \nabla_u \Psi(t,x,u;\lambda,\upsilon,\mu,\nu)+\cN^L_U(u),\\[2pt]
\lambda\geq0,\ \lambda_i=0\ \forall i\in I_t^-(x,u),\nonumber\\[2pt]
\mu_i=0\ \forall i\in \cI^{+0}_t(x,u),\ \nu_i=0\ \forall i\in \cI^{0+}_t(x,u),\\[2pt]
 \mu_i>0, \nu_i> 0\ {\rm or}\  \mu_i\nu_i=0\ \forall i\in \cI^{00}_t(x,u),
\end{array}\right.\\[4pt]
&& \Longrightarrow (\lambda,\upsilon,\mu,\nu)=0,
\end{eqnarray}
which clearly implies the MPEC quasi-normality defined in Theorem \eqref{thmm}(c) and hence the local error bound condition for the system representing $S(t)$ as in (\ref{cpc1}) holds at $(x,u)$ by Proposition \ref{error bound}. Thus, by Proposition \ref{prop-auto}, we can have the following result  immediately.

\begin{cor}\label{cor-auto}
Let all the mappings $g(\cdot),h(\cdot),G(\cdot),H(\cdot),U(\cdot)$ be autonomous. Assume that $C^{\epsilon,R}_*$ is compact for some $\epsilon>0$ and condition \eqref{wbcqnew} holds for every $(x,u)$ such that $(t,x,u)\in C^{\epsilon,R}_*$.  Then there exists certain positive constant $\pi$ such that for every $t\in[t_0,t_1]$, the bounded slope condition (\ref{bsc}) holds with  $k_S(t)=\pi$.
\end{cor}

In Proposition \ref{prop-auto} and Corollary \ref{cor-auto}, conditions \eqref{wbcq} and \eqref{wbcqnew} are both required to hold over some neighborhood of the optimal process $(x^*(\cdot),u^*(\cdot))$. In order to weaken this requirement, Clarke and De Pinho \cite[Definition 4.7]{C-P}  introduced the following concept.

\begin{defi}\label{defi admiss}
We say that $(t, x^*(t),u)$ is an admissible cluster point of $(x^*(\cdot),u^*(\cdot))$ if there exist a sequence $\{t^k\}\subseteq [t_0,t_1]$ converging to $t$ and corresponding points $(x^k,u^k)\in S(t^k)$ such that $\lim\limits_{k\to\infty}x^k\to x^*(t)$ and $\lim\limits_{k\to\infty}u^k=\lim\limits_{k\to\infty}u^*(t^k)=u$.
\end{defi}

Based on Definition \ref{defi admiss}, we have the following sufficient condition for the bounded slope condition to hold with certain positive constant.

\begin{prop}\label{bsc-auto}
Let $R(\cdot)\equiv r>0$ be a positive constant function and  all the mappings $g(\cdot)$, $h(\cdot)$, $G(\cdot)$, $H(\cdot)$, $U(\cdot)$ autonomous. Assume that for all admissible cluster points $(t,x^*(t),u)$ of $(x^*(\cdot),u^*(\cdot))$, condition (\ref{wbcq}) and the local error bound condition for the system representing $S(t)$ as in (\ref{cpc1}) hold at $(x^*(t),u)$ or the stronger condition \eqref{wbcqnew} holds at $(x^*(t),u)$.  Then for every $t\in[t_0,t_1]$, the bounded slope condition (\ref{bsc}) holds with some radius $\eta\in (0,r)$ and $k_S(t)= \pi$ for some constant $\pi>0$.
\end{prop}
\begin{proof}
Mimicking the proof of Proposition \ref{prop-auto}, we can show that there exist $\epsilon_1\in(0,\epsilon)$, $\eta\in (0,r)$, and $\pi>0$ such that for every $t\in [t_0,t_1]$, the following bounded slope condition holds:
\[
(x,u)\in S_*^{\epsilon_1,\eta}(t),\ (\alpha,\beta)\in {\cal N}^P_{S}(x,u)\Longrightarrow \|\alpha\|\leq \pi\|\beta\|.
\]
The proof is complete.
\end{proof}

Although all the mappings $g(\cdot),h(\cdot),G(\cdot),H(\cdot), U(\cdot)$ are assumed to be autonomous in Propositions \ref{prop-auto}--\ref{bsc-auto} and Corollary \ref{cor-auto}, their results can be applied to the non-autonomous case if $U(t)\equiv U$ is autonomous and we treat the time variable $t$ as a state variable. We now illustrate how this can be done. Since
\begin{eqnarray}\label{sigma}
\sigma(t)=t\quad \forall  t\in[t_0,t_1] \Longleftrightarrow \dot{\sigma}(t)=1\ \forall t\in [t_0,t_1],\ \sigma(t_0)=t_0,
\end{eqnarray}
it is clear that the OCPEC is equivalent to the following optimal control problem:
\begin{eqnarray*}\label{auto}
\min  &&  J(x(\cdot),u(\cdot)):=\int_{t_0}^{t_1} F(\sigma(t), x(t), u(t)) dt + f(x(t_0),x(t_1)),\nonumber \\
{\rm s.t.} && \dot{x}(t)=\phi(\sigma(t),x(t), u(t)),\ \dot{\sigma}(t)=1 \,\quad {\rm a.e.} \,t\in [t_0,t_1],\nonumber\\
&& g(\sigma(t), x(t), u(t) ) \leq 0,\ h(\sigma(t), x(t), u(t) )=0 \quad {\rm a.e.} \,t\in [t_0,t_1],\\
&& 0\leq G(\sigma(t),x(t), u(t))\perp H(\sigma(t),x(t), u(t))\geq0\quad {\rm a.e.} \,t\in [t_0,t_1],
\\
&& u(t) \in U \quad {\rm a.e.} \,t\in [t_0,t_1],\nonumber\\
&& \sigma(t_0)=t_0,\quad (x(t_0), x(t_1))\in E.\nonumber
\end{eqnarray*}
It is easy to see that $(\sigma(\cdot),x^*(\cdot),u^*(\cdot))$ is a  local minimizer of radius $R(\cdot)$ for the above problem if $(x^*(\cdot),u^*(\cdot))$ is a local minimizer of radius $R(\cdot)$ for the OCPEC and $\sigma(\cdot)$ is defined in (\ref{sigma}). Thus, the results in Propositions \ref{prop-auto}--\ref{bsc-auto} and Corollary \ref{cor-auto} can be applied to the above problem to get the desired result.

We close this section by noting the equivalence of the S-stationarity for the OCPEC and the classical necessary optimality condition for the OCPEC treated as an optimal control problem with mixed inequality constraints  (\ref{cpcp}). The proof for the following result is similar to \cite[Proposition 4.1]{Fletcher06} and we omit the proof here.

\begin{prop}
$(x^*(\cdot),u^*(\cdot))$ is an FJ type stationary solution of the OCPEC treated as an optimal control problem with mixed inequality constraints  (\ref{cpcp}) if and only if $(x^*(\cdot),u^*(\cdot))$ is an FJ type S-stationary solution of the OCPEC for which those two sets of multipliers can be chosen as the same.
\end{prop}

\section{A simple example}

In this section, we consider a simple class of the OCPEC in which the ``best" control needs to be chosen from the DCP (\ref{dcp1}) so as to make the final state $x(t_1)$ reach some prescribed target ${\cal T}$ from a given initial state $x(t_0)$. Mathematically, the problem considered in this section is
\begin{eqnarray}\label{simple model}
\begin{array}{rl}
\min& \frac{1}{2}\|x(t_1)-{\cal T}\|^2\\[3pt]
& \dot{x}(t)=\phi(t,x(t), u(t))  \quad{\rm a.e.}  \ t\in [t_0,t_1], \\[3pt]
& 0\leq u(t) \perp \Upsilon(t,x(t), u(t)) \geq 0  \quad{\rm a.e.}  \ t\in [t_0,t_1],\\[3pt]
&  x(t_0)\in E_0,\\
\end{array}
\end{eqnarray}
where $E_0\subseteq\Re^n$ is a closed subset. In this case, $S(t):=\{(x,u):(u,\Upsilon(t,x,u))\in\cC^m\}$ and $S_*^{\epsilon,R}(t)$ is defined as in \eqref{neighbor}.
%
For simplicity, we assume that the functions $\phi(\cdot),\Upsilon(\cdot)$ are ${\cal L}$ measurable in variable $t$ and strictly differentiable in $(x,u)$.  Moreover, there exist measurable functions $k_x^\phi(\cdot),k_u^\phi(\cdot),k_x^\Upsilon(\cdot)$ such that for almost every $t\in[t_0,t_1]$,
\begin{equation*}
\begin{array}{l}
\|\Upsilon(t,x^1,u)-\Upsilon(t,x^2,u)\|\leq k_x^\Upsilon(t)\|x^1-x^2\|\quad \forall (x^1,u),(x^2,u)\in S_*^{\epsilon,R}(t),\\[3pt]
\|\phi(t,x^1,u^1)-\phi(t,x^2,u^2)\|\leq k_x^\phi(t)\|x^1-x^2\|+ k_u^\phi(t)\|u^1-u^2\|\quad \forall (x^1,u^1),(x^2,u^2)\in S_*^{\epsilon,R}(t).
\end{array}
\end{equation*}

In the following, we apply the derived results in Section \ref{section3} to problem \eqref{simple model}. The following result follows immediately from Proposition \ref{prop-strongcon} and Theorems \ref{thmw}--\ref{thmS}. Note that since the final point $x(t_1)$ in problem (\ref{simple model}) is free, $\lambda_0$ can be chosen as 1 by Remark \ref{rek}. Moreover, $k_S(t)=\kappa(t)k_x^\Upsilon(t)\ {\rm a.e.}\ t\in[t_0,t_1] $ in this case, and since $(x^*(\cdot),u^*(\cdot))$  is feasible to problem (\ref{simple model}), we have
\[
u^*(t)\in\{u:(u,\Upsilon(t,x^*(t),u))\in\cC^m\}\quad {\rm a.e.}\ t\in[t_0,t_1].
\]

\begin{thm}
Let $(x^*(\cdot),u^*(\cdot))$ be a  local minimizer of radius $R(\cdot)$ for problem (\ref{simple model}). Assume that there exists a measurable function $\kappa(\cdot)$ such that for almost every $t\in[t_0,t_1]$,
\begin{eqnarray*}
\left\{\begin{array}{l}
(x,u)\in S_*^{\epsilon,R}(t),\\[2pt]
 \mu_i=0\ \forall i\in \cI^{+0}_t(x,u),\ \nu_i=0\ \forall i\in \cI^{0+}_t(x,u),\\[2pt]
  \mu_i>0,\ \nu_i> 0\ {\rm or}\  \mu_i\nu_i=0\ \forall i\in \cI^{00}_t(x,u),
\end{array}\right.
\Longrightarrow
\|(\mu,\nu) \|\leq \kappa(t)\|\mu+\nabla_{u} \Upsilon (t,x,u)\nu\|.
\end{eqnarray*}
Assume also that the functions $k_x^\phi(\cdot),\kappa(\cdot) k_x^\Upsilon(\cdot) k_u^\phi(\cdot)$ are integrable and there exists a positive number $\eta>0$ such that $R(t)\geq\eta k_S(t)\ a.e.\ t\in[t_0,t_1]$.
Then $(x^*(\cdot),u^*(\cdot))$ is a W-stationary point, i.e., there exist an arc $p(\cdot)$ and measurable functions $\lambda^G(\cdot)$, $\lambda^H(\cdot)$ such that for almost every $t\in[t_0,t_1]$,
\begin{itemize}
\item[\rm 1)] $p(t_0)\in {\cal N}_{E_0}^L(x^*(t_0)), \ -p(t_1)=x^*(t_1)-{\cal T}$,
\item[\rm 2)] $(\dot{p}(t),0)=-\nabla_{x,u}\phi(t,x^*(t),u^*(t))p(t) -(0,\lambda^G(t))-\nabla_{x,u}\Upsilon(t,x^*(t),u^*(t))\lambda^H(t)$,\\[4pt]
      $\lambda^G_i(t)=0\ \forall  i\in \cI^{+0}_t(x^*,u^*),\ \lambda^H_i(t)=0\ \forall  i\in \cI^{0+}_t(x^*,u^*)$,
\item[\rm 3)] $(x^*(t), u)\in S(t), \|u-u^*(t))\|< R(t)\Rightarrow \langle p(t),\phi(t,x^*(t),u) \rangle  \leq \langle p(t), \phi (t,x^*(t), u^*(t)) \rangle.$
\end{itemize}

If, in addition, the MPEC quasi-normality holds at $u^*(t)\in\{u:(u,\Upsilon(t,x^*(t),u))\in\cC^m\}$ for almost every $t\in[t_0,t_1]$, then $(x^*(\cdot),u^*(\cdot))$ is an M-stationary point, i.e., all the above results 1), 2) and 3) hold and there exist
measurable functions $\eta^G(\cdot),\eta^H(\cdot)$ such that
\begin{eqnarray*}
&&\nabla_{u}\phi(t,x^*(t),u^*(t))p(t) +\eta^G(t)+\nabla_{u}\Upsilon(t,x^*(t),u^*(t))\eta^H(t)=0,\nonumber\\
&& \eta^G_i(t)=0\ \forall i\in \cI^{+0}_t(x^*,u^*),\ \eta^H_i(t)=0\ \forall i\in \cI^{0+}_t(x^*,u^*),\\
&& \eta^G_i(t)>0,\eta^H_i(t)>0\ {\rm or}\ \eta^G_i(t)\eta^H_i(t)=0\ \forall i\in \cI^{00}_t(x^*,u^*).
\end{eqnarray*}

If, in addition, the MPEC-LICQ holds at $u^*(t)\in\{u:(u,\Upsilon(t,x^*(t),u))\in\cC^m\}$ for almost every $t\in[t_0,t_1]$, then
$(x^*(\cdot),u^*(\cdot))$ is an S-stationary point, that is,  for almost every $t\in[t_0,t_1]$, all the above results 1), 2) and 3) hold and
\[
\lambda^G_i(t)\geq0,\ \lambda^H_i(t)\geq0 \ \forall i\in \cI^{00}_t(x^*,u^*).
\]
\end{thm}

\bigskip

In the rest of this section, we focus on a proper specialization of the DCP (\ref{dcp1}) where
\[
\phi(t,x,u):=Ax+Bu+c,\quad \Upsilon(t,x,u):=Cx+Du+q
\]
for some  matrices $A,B,C,D$ and vectors $c,q$ of appropriate sizes. Note that in this case, the functions $\phi(\cdot),\Upsilon(\cdot)$ are both autonomous. Then problem (\ref{simple model}) reduces to
\begin{eqnarray}\label{simple model2}
\begin{array}{rl}
\min\limits& \frac{1}{2}\|x(t_1)-{\cal T}\|^2\\[3pt]
& \dot{x}(t)= Ax(t)+Bu(t)+c \quad{\rm a.e.}  \ t\in [t_0,t_1], \\[3pt]
& 0\leq u(t) \perp Cx(t)+Du(t)+q \geq 0  \quad{\rm a.e.}  \ t\in [t_0,t_1],\\[3pt]
&  x(t_0)\in E_0.\\
\end{array}
\end{eqnarray}
In this case, let $R(\cdot)\equiv r>0$ be a positive constant function and set
\begin{eqnarray}
&&S(t)\equiv S:=\{(x,u): (u, Cx+Du+q)\in\cC^m\},\label{special6}\\
&&S_*^{\epsilon,r}(t):=\big\{(x,u)\in S: \|x-x^*(t)\|\leq\epsilon, \|u-u^*(t)\|\leq r\big\},\nonumber\\
&&C^{\epsilon,r}_*:={\rm cl}\,\{(t,x,u)\in[t_0,t_1]\times \Re^n\times\Re^m: (x,u)\in S^{\epsilon,r}_*(t)\}.\label{special5}
\end{eqnarray}

The following result follows from Proposition \ref{prop-auto} and Theorems \ref{thmm}--\ref{thmS} immediately. Note that the local error bound condition for the system representing $S(t)$ as in (\ref{special6}) holds since the functions $\phi(t,\cdot,\cdot),\Upsilon(t,\cdot,\cdot)$ are affine in $(x,u)$, and when $u^*(\cdot)$ is bounded over $[t_0,t_1]$, the compactness of $C^{\epsilon,r}_*$ is superfluous. Moreover, since $(x^*(\cdot),u^*(\cdot))$ is feasible to problem (\ref{simple model2}), we have
\[
u^*(t)\in\{u:(u,C x^*(t)+Du+q)\in\cC^m\}\quad {\rm a.e.}\ t\in[t_0,t_1].
\]

\begin{thm}
Let $(x^*(\cdot),u^*(\cdot))$ be a  local minimizer of positive constant radius $R(\cdot)\equiv r>0$ for problem (\ref{simple model2}) and $C^{\epsilon,r}_*$ defined in (\ref{special5}) be compact. Assume that for all $(t,x,u)\in C^{\epsilon,r}_*$,
\begin{eqnarray*}
\left\{\begin{array}{l}
 \mu+ D^T\nu=0,\\[2pt]
 \mu_i=0\ \forall i\in \cI^{+0}_t(x,u),\ \nu_i=0\ \forall i\in \cI^{0+}_t(x,u),\\[2pt]
  \mu_i>0,\ \nu_i> 0\ {\rm or}\  \mu_i\nu_i=0\ \forall i\in \cI^{00}_t(x,u),
\end{array}\right.
\Longrightarrow
C^\top\mu=0.
\end{eqnarray*}
Then $(x^*(\cdot),u^*(\cdot))$ is an M-stationary point, i.e.,  there exist an arc $p(\cdot)$ and measurable functions $\lambda^G(\cdot)$, $\lambda^H(\cdot)$ and $\eta^G(\cdot)$, $\eta^H(\cdot)$ such that for almost every $t\in[t_0,t_1]$,
\begin{itemize}
\item[\rm 1)] $p(t_0)\in {\cal N}_{E_0}^L(x^*(t_0)), \ -p(t_1)=x^*(t_1)-{\cal T}$,
\item[\rm 2)] $-A^\top p(t)-C^\top \lambda^H(t)=\dot{p}(t),\  B^\top p(t)+\lambda^G(t) + D^\top \lambda^H(t)=0$,\\[3pt]
      $\lambda^G_i(t)=0\ \forall i\in \cI^{+0}_t(x^*,u^*),\ \lambda^H_i(t)=0\ \forall i\in \cI^{0+}_t(x^*,u^*)$,
\item[\rm 3)] $ \langle B^\top p(t), u-u^*(t) \rangle \leq0\quad \forall u: (x^*(t), u)\in S, \|u-u^*(t))\|< r$,

\item[\rm 4)] $B^\top p(t) + \eta^G(t) + D^\top \eta^H(t)=0$,\\[3pt]
$\eta^G_i(t)=0\ \forall i\in \cI^{+0}_t(x^*,u^*),\ \eta^H_i(t)=0\ \forall i\in \cI^{0+}_t(x^*,u^*)$,\\[3pt]
$\eta^G_i(t)>0,\eta^H_i(t)>0\ {\rm or}\ \eta^G_i(t)\eta^H_i(t)=0\ \forall i\in \cI^{00}_t(x^*,u^*)$.
\end{itemize}

If, in addition, the MPEC-LICQ holds at $u^*(t)\in\{u:(u,C x^*(t)+Du+q)\in\cC^m\}$ for almost every $t\in[t_0,t_1]$, then
$(x^*(\cdot),u^*(\cdot))$ is an S-stationary point, that is,  for almost every $t\in[t_0,t_1]$, all the above results 1), 2) and 3) hold and
\[
\lambda^G_i(t)\geq0,\ \lambda^H_i(t)\geq0 \ \forall i\in \cI^{00}_t(x^*,u^*).
\]
\end{thm}


\bigskip

\noindent\textbf{Acknowledgements.} The authors are grateful to the two
anonymous referees for their extremely helpful comments and suggestions.


\end{document}